\documentclass[12pt]{article}

\usepackage{fancyhdr}
\usepackage{enumitem} 
\newlist{condenum}{enumerate}{1}
\usepackage{xcolor} 
\usepackage{graphicx} 
\usepackage{hyperref} 

\usepackage{amsmath}
\allowdisplaybreaks
\usepackage{amsthm}
\usepackage{amssymb}
\usepackage{mathtools}
\usepackage{mathabx}
\usepackage{bbm}
\usepackage{shuffle}
\usepackage[mathscr]{eucal}
\usepackage{stmaryrd} 
\usepackage{makecell}

\usepackage{tikz}
\usetikzlibrary{patterns, arrows, matrix, positioning}
\usepackage{tikz-cd}

\usepackage[
	backend = biber,
	style=numeric,
	natbib=false,
	isbn=false, 
	url = false,
	isbn=false, 
	doi=false
]{biblatex}

\renewbibmacro{in:}{\ifentrytype{article}{}{\printtext{\bibstring{in}\intitlepunct}}}
\DeclareFieldFormat[article, book, incollection]{pages}{#1}

\newtheorem{thm}[equation]{Theorem}
\newtheorem{conj}[equation]{Conjecture}
\newtheorem{question}[equation]{Open Problem}
\newtheorem{prop}[equation]{Proposition}

\newtheorem{cor}[equation]{Corollary}
\theoremstyle{definition}

\theoremstyle{remark}
\newtheorem{rem}[equation]{Remark}
\newtheorem{rems}[equation]{Remarks}
\numberwithin{equation}{section}
\newcommand{\FF}{\mathbb{F}}
\newcommand{\CC}{\mathbb{C}}

\newcommand{\ZZ}{\mathbb{Z}}

\newcommand{\KK}{\mathbb{K}}
\def\mydefb#1{\expandafter\def\csname #1#1#1\endcsname{\mathcal{#1}}}
\def\mydefallb#1{\ifx#1\mydefallb\else\mydefb#1\expandafter\mydefallb\fi}
\mydefallb ABCDEFGHIJKLMNOPQRSTUVWXYZ\mydefallb

\renewcommand{\epsilon}{\varepsilon}
\newcommand{\spanning}{\operatorname{-span}}

\renewcommand{\setminus}{-}

\let\shortto\to
\renewcommand{\to}{\longrightarrow}
\let\shortmapsto\mapsto
\renewcommand{\mapsto}{\longmapsto}

\newcommand{\UT}{\mathrm{UT}}
\newcommand{\GL}{\mathrm{GL}}
\newcommand{\Mat}{\mathrm{Mat}}
\newcommand{\Sym}{{\mathrm{Sym}}}
\newcommand{\QSym}{{\mathcal Q}\Sym}
\newcommand{\cf}{\mathsf{cf}}
\newcommand{\scf}{\mathsf{scf}}
\newcommand{\cfunsupp}{\cf^{\operatorname{uni}}_{\operatorname{supp}}}
\newcommand{\cfunchar}{\cf^{\operatorname{uni}}_{\operatorname{char}}}

\newcommand{\ket}[1]{#1 \big\rangle}
\newcommand{\cano}{\operatorname{cano}}
\newcommand{\USind}{(n)}
\newcommand{\zetaUS}{ \delta_{(\bullet)}^{\ast}}
\newcommand{\canoUS}{\mathbf{p}_{\{1\}}}
\newcommand{\canoUC}{\mathbf{p}_{\mathbbm{1}}}
\newcommand{\UTnInd}{([n], \emptyset)}
\newcommand{\zetaCQSbullet}{(q-1)^{\bullet} \delta_{([\bullet], \emptyset)}^{\ast}}
\newcommand{\canoCQS}{\mathbf{c}_{\{1\}}}
\newcommand{\canoLLT}{\mathbf{c}_{\mathbbm{1}}}
\newcommand{\ind}{\operatorname{Ind}}
\newcommand{\res}{\operatorname{Res}}

\newcommand{\permchar}{\overline{\chi}}
\newcommand{\permind}{\overline{\delta}}
\newcommand{\reg}{\operatorname{reg}}

\newcommand{\asc}{\operatorname{asc}}

\newcommand{\Dyck}{\operatorname{Dyck}}
\newcommand{\diag}{\operatorname{Diag}}
\newcommand{\mesa}{\operatorname{Mesa}}

\begin{document}

\title{A unipotent realization of the chromatic quasisymmetric function}
\author{Lucas Gagnon}
\date{\today}
\maketitle
\begin{abstract}
This paper realizes of two families of combinatorial symmetric functions via the complex character theory of the finite general linear group $\mathrm{GL}_{n}(\mathbb{F}_{q})$: chromatic quasisymmetric functions and vertical strip LLT polynomials.  
The associated $\mathrm{GL}_{n}(\mathbb{F}_{q})$ characters are elementary in nature and can be obtained by induction from certain well-behaved characters of the unipotent upper triangular groups $\mathrm{UT}_{n}(\mathbb{F}_{q})$.  
The proof of these results also gives a general Hopf algebraic approach to computing the induction map.  
Additional results include a connection between the relevant $\mathrm{GL}_{n}(\mathbb{F}_{q})$ characters and Hessenberg varieties and a re-interpretation of known theorems and conjectures about the relevant symmetric functions in terms of $\mathrm{GL}_{n}(\mathbb{F}_{q})$.  
\end{abstract}


\section{Introduction}

The chromatic symmetric function sits at a nexus of disparate areas of mathematics.  
At face value, this symmetric function encodes the coloring problem of a graph as an analogue of the chromatic polynomial~\cite{Sta}.  
However, through a well-known equivalence between the ring of symmetric functions and the representation theory of the symmetric groups (see e.g.~\cite{Mac}), some chromatic symmetric functions are also complex characters of the symmetric group~\cite{Gash}.   
Moreover, by way of a $t$-analogue known as the chromatic quasisymmetric function, Brosnan and Chow~\cite{BrosChow} and Guay-Paquet~\cite{GP16} independently proved that the characters corresponding to indifference graphs are afforded by symmetric group representations on the cohomology rings of regular semisimple Hessenberg varieties, as predicted by a conjecture of Shareshian and Wachs~\cite{ShWa}.  
Thus, certain questions about graphs, representation theory, and algebraic geometry coincide in the combinatorics of these symmetric functions, and vice versa.

At about the same time, a sequence of superficially unrelated developments occurred in the character theory of the group of unipotent upper triangular matrices $\UT_{n}$ over a finite field $\FF_{q}$.  
Unlike the symmetric group, the conjugacy classes and irreducible characters of $\UT_{n}$ are exceptionally complicated and cannot be described with modern combinatorial tools~\cite{GudEtAl}.  
However, beginning with the work of Andr\'{e}~\cite{Andre}, a theory of well-behaved reducible characters---known a supercharacters---has developed, leading to a combinatorial representation theory of $\UT_{n}$ without irreducibles, as in~\cite{AgBerTh, AgEtAl}.  
A recent example given by Aliniaeifard and Thiem~\cite{AlTh20} constructs supercharacters which are imbued with Catalan combinatorics coming from a family of normal subgroups of $\UT_{n}$.  
These same subgroups and supercharacters will appear in this paper, where they will be indexed in a canonical way by indifference graphs.

This paper uses the representation theory of the general linear group $\GL_{n}$ over $\FF_{q}$ to establish a connection between the supercharacter theory of $\UT_{n}$ and the chromatic (quasi-) symmetric function.  
Both $\UT_{n}$ and its subgroups are contained in $\GL_{n}$.    
The main result, Theorem~\ref{thm:CQSbijection} shows that up to a factor of $(q-1)^{n}$, inducing the trivial character from each of these subgroups gives a map
\[
\left\{ \begin{array}{c} \text{indifference graph} \\ \text{indexed subgroups} \end{array}\right\} \xrightarrow{\;\; \ind^{\GL_{n}}_{(-)}(\mathbbm{1}) \;\;} \left\{ \begin{array}{c}  \text{chromatic quasisymmetric functions for} \\ \text{indifference graphs evaluated at $t = q$} \end{array} \right\},
\]
using an implicit identification between characters of $\GL_{n}$ with unipotent support and symmetric functions coming from the Hall algebra; more details can be found in Section~\ref{sec:mainresult1}.  
This result is a $\GL_{n}(\FF_{q})$-analogue of the Brosnan--Chow--Guay-Paquet theorem, in which cohomology rings are replaced by a permutation representation on the cosets of certain unipotent subgroups.

The remaining sections of the paper explore the implications of the main result for the theory of chromatic quasisymmetric functions.
Many of these consequences are reminiscent of consequences of the Brosnan--Chow--Guay-Paquet theorem.  
Along with Theorem~\ref{thm:CQSbijection} itself, these similarities come as a surprise, especially since the association between characters of $\GL_{n}$ and symmetric functions used above is markedly different from the classical association for the symmetric groups.  
Intuition notwithstanding, however, each result appears to be straightforward, or even inevitable once the right perspective is achieved.

Section~\ref{sec:Hessenberg} relates Theorem~\ref{thm:CQSbijection} to the study of Hessenberg varieties, but not the ones appearing in the Brosnan--Chow--Guay-Paquet theorem.  
Instead, the values of the $\GL_{n}$ characters in Theorem~\ref{thm:CQSbijection} count the points of a nilpotent Hessenberg variety over $\FF_{q}$ associated to an ad-nilpotent ideal.
The analogous complex Hessenberg varieties over have been studied by Precup and Sommers~\cite{PrecupSommers}, who found an independent connection to the chromatic quasisymmetric function via Poincar\'{e} polynomials.  
Corollary~\ref{cor:adnilpotentpoincare} links these results by showing that the Poincar\'{e} polynomials for the complex Hessenberg varieties also count the points of the corresponding Hessenberg variety over $\FF_{q}$.

The chromatic quasisymmetric functions of indifference graphs are also closely related to another family of symmetric functions known as unicellular LLT polynomials~\cite{CM} (see also~\cite{HHL}), and Section~\ref{sec:mainresult2} reframes this relationship as a $\GL_{n}$ representation theoretic one.  
There is a second, more standard realization of symmetric functions as \textit{unipotent} characters of $\GL_{n}$, and up to a twist by the involution $\omega$, Theorem~\ref{thm:LLTbijection} gives a map
\[
\left\{ \begin{array}{c} \text{indifference graph} \\ \text{indexed subgroups} \end{array}\right\} \xrightarrow{\;\; \operatorname{proj}_{\text{unipotent}} \circ \ind^{\GL_{n}}_{(-)}(\mathbbm{1}) \;\;}  \left\{ \begin{array}{c}  \text{unicellular LLT polyno-} \\ \text{mials evaluated at $t = q$} \end{array} \right\},
\]
where $\operatorname{proj}_{\text{unipotent}}$ is the operation which replaces a character of $\GL_{n}$ with the sum of its irreducible unipotent constituents.  
In fact, by applying the composite map to additional characters of $\UT_{n}$---including supercharacters---Theorem~\ref{thm:LLTbijection} finds the larger family of vertical strip LLT polynomials as unipotent characters of $\GL_{n}$.  
These symmetric functions are known to appear in the representation theory of quantum groups~\cite{LLT}, affine Hecke algebras~\cite{GrojHai}, and the symmetric groups~\cite{GP16, HHL}, but this is their first appearance in the representation theory of $\GL_{n}$.

Finally, both chromatic quasisymmetric functions and LLT polynomials 
are the subject of ``positivity conjectures'' which are at least partially open.  
Such a conjecture postulates that when a particular symmetric function is expressed in a chosen basis, the coefficient of each basis element will be a polynomial in $t$ with nonnegative coefficients.  
For chromatic quasisymmetric functions, the modified Stanley--Stembridge conjecture~\cite[Conjecture 1.3]{ShWa} (see also~\cite{StaStem}) concerns the elementary basis, and is almost entirely open.  
For LLT polynomials, positivity in the Schur basis has been established by Grojnowski and Haiman~\cite{GrojHai}, but no ``positive'' combinatorial formula is known in general~\cite{HaglundBook}.  
Section~\ref{sec:positivity} describes the meaning of these conjectures---and one more, recently resolved by D'Adderio~\cite{Dadd} and Alexandersson and Sulzgruber~\cite{AS}---in $\GL_{n}$ representation theory.  
This does not lead to immediate progress on any conjecture, but it may be a useful guide for future work.

The method of proof for Theorems~\ref{thm:CQSbijection} and~\ref{thm:LLTbijection} may also be of independent interest.  
At a high level, I am able to translate Guay--Paquet's proof in~\cite{GP16} into the the (super-)character theory of $\UT_{n}$ and $\GL_{n}$ in such a way that both results follow immediately.  
However, this translation also gives a more general Hopf algebraic conduit from the combinatorial representation theory of $\UT_{n}$ to that of $\GL_{n}$.  
Since matters of $\UT_{n}$ character theory are usually very difficult, the tractability of this approach alone is a significant development.  
Furthermore, these results begin to answer lingering questions from the paper~\cite{AgEtAl} about the Hopf algebraic enumerative invariants of certain supercharacters of $\UT_{n}$.  

A short summary of the aforementioned framework and the machinery of~\cite{GP16} is given in this paragraph.  
In~\cite{AgBerSot}, Aguiar, Bergeron, and Sottile constructively classify all Hopf algebra homomorphisms from an arbitrary Hopf algebra to the Hopf algebra of symmetric functions $\Sym$ with the linear characters of the domain.  
This generalizes Zelevinsky's theory of PSH algebras, which completely describes the character theory of $\GL_{n}$ by constructing a collection of homomorphisms from a Hopf algebra $\cf(\GL_{\bullet})$ of $\GL_{n}$-class functions to $\Sym$, the ring of symmetric functions.  
In~\cite{Gaga}, I construct an analogous Hopf algebra $\cf(\UT_{\bullet})$ on the class functions of $\UT_{n}$, and show that induction $\ind^{\GL_{n}}_{\UT_{n}}$ induces a Hopf algebra homomorphism to $\cf(\GL_{\bullet})$.  
By composing induction with any of Zelevinsky's maps to $\Sym$, the classification of~\cite{AgBerSot} can be used to describe the induction map itself, and Theorems~\ref{thm:canoCQS} and~\ref{thm:canoLLT} do so.  
The classification of~\cite{AgBerSot} was also used in~\cite{GP16} to construct the chromatic quasisymmetric function using a Hopf algebra structure on Hessenberg varieties, and I show that this coincides with induction of Catalan supercharacters and related objects.

This Hopf algebraic approach builds on the previously understood relationship between the combinatorics of unipotent subgroups and of finite groups of Lie type, including $\GL_{n}$~\cite{AnTh, GelGra, Kawanaka, Zel}.  
Future work should continue to push this connection: it may be possible to transplant some of the framework in this paper and~\cite{Gaga} into other Lie types.  
In doings so, one might find the generalized LLT polynomials defined in~\cite{GrojHai}, yet-to-be-discovered variants of the chromatic quasisymmetric function, and more nilpotent Hessenberg varieties.  

The remainder of the paper is organized as follows.  Section~\ref{sec:prelims} describes the general background material for the paper.  Section~\ref{sec:mainresult1} concerns Theorem~\ref{thm:CQSbijection} and the chromatic quasisymmetric function, and Section~\ref{sec:Hessenberg} relates these results to Hessenberg varieties.  Section~\ref{sec:mainresult2} concerns Theorem~\ref{thm:LLTbijection} and the vertical strip LLT polynomial, and is essentially independent of Sections~\ref{sec:mainresult1} and~\ref{sec:Hessenberg}.  Finally, Section~\ref{sec:positivity} connects my results to various positivity conjectures.


\paragraph{Acknowledgments}

Along with~\cite{Gaga}, this paper is part of a Ph.D.~thesis undertaken at the University of Colorado Boulder, and I am extremely thankful for the support and insights of my advisor Nathaniel Thiem throughout this process.  
I am also grateful to Martha Precup for a helpful conversation about Hessenberg varieties.
An extended abstract of this work will appear in the proceedings of the FPSAC 2023 conference.

\section{Preliminaries}
\label{sec:prelims}

This section gives the shared preliminary material for Sections~\ref{sec:mainresult1} and~\ref{sec:mainresult2}.  This includes definitions of each of the relevant Hopf algebras, background material from representation theory and combinatorics, and a short review of the theory of combinatorial Hopf algebras.


\subsection{Hopf algebras and (Quasi-)symmetric functions}
\label{sec:QSym}

This section will describe the Hopf algebras of quasisymmetric and symmetric functions, and their role as universal objects in the theory of combinatorial Hopf algebras.  Throughout this paper, the term ``Hopf algebra'' will refer to a graded connected Hopf algebra over the field of complex numbers $\CC$, and all homomorphims and sub-Hopf algebras are graded.  

A \emph{composition} of $n \in \ZZ_{\ge 0}$ is a finite (possibly empty) sequence of positive integers $\alpha = (\alpha_{1}, \ldots, \alpha_{k})$ with $\alpha_{1} + \cdots + \alpha_{k} = n$.  Call each $\alpha_{i}$ a \emph{part} of $\alpha$, and write $\ell(\alpha) = k$ for the number of parts of $\alpha$.  The \emph{monomial quasisymmetric function} associated to the composition $\alpha$ is 
\[
M_{\alpha} = \sum_{i_{1} < \cdots < i_{\ell(\alpha)}} x_{i_{1}}^{\alpha_{1}} x_{i_{2}}^{\alpha_{2}} \cdots x_{i_{\ell(\alpha)}}^{\alpha_{\ell(\alpha)}} \in \CC[[\mathbf{x}]].
\]
where $\mathbf{x} = \{x_{1}, x_{2}, \ldots \}$ is an infinite, totally ordered set of indeterminates.  The \emph{Hopf algebra of quasisymmetric functions} is the graded commutative, noncocommutative Hopf algebra 
\[
\QSym= \CC\spanning\{M_{\alpha} \;|\; \text{$\alpha$ is a composition}\}.
\]
The product of $\QSym$ is inherited from $\CC[[\mathbf{x}]]$ and the coproduct is given by deconcatenation:
\[
\Delta(M_{\alpha}) = \hspace{-0.5em} \sum_{\ell(\alpha) \ge k \ge 0} \hspace{-0.5em} M_{(\alpha_{1}, \ldots, \alpha_{k})} \otimes M_{(\alpha_{k}, \ldots, \alpha_{\ell(\alpha)})}.
\]

A \emph{partition of $n$} is a composition of $n$ is with nonincreasing parts. Let 
\[
\PPP = \bigsqcup_{n \ge 0} \PPP(n) \qquad \text{with}\qquad \PPP(n) = \{\text{partitions of $n$}\}.
\]
The \emph{Hopf algebra of symmetric functions} is the maximal cocommatative Hopf subalgebra
\[
\Sym = \CC\spanning\{m_{\lambda}\;|\; \lambda \in \PPP\} \subseteq \QSym \qquad \text{with}\qquad m_{\lambda} = \hspace{-0.5em} \sum_{\operatorname{sort}(\alpha) = \lambda} \hspace{-0.5em} M_{\alpha},
\]
where $\operatorname{sort}(\alpha)$ denotes the partition obtained by listing the parts of $\alpha$ in nonincreasing order.  

Two additional bases of $\Sym$ will be used in later sections.  The first basis consists of the \emph{elementary symmetric functions} $\{e_{\lambda} \;|\; \lambda \in \PPP\}$ defined by
\[
e_{\lambda} = e_{\lambda_{1}} \cdots e_{\lambda_{\ell}} \qquad\text{where}\qquad e_{k} = m_{(1^{k})}.
\]
The second basis comprises the \emph{Schur functions} $\{s_{\lambda} \;|\; \lambda \in \PPP\}$, which I will not define; see~\cite[\nopp I.3]{Mac}.

%
%

The antipode of $\Sym$ acts as $(-1)^{n} \omega$ on the $n$th graded component, where $\omega$ is the involutive automorphism of $\Sym$ defined in~\cite[\nopp I.4]{Mac}, given by $\omega(s_{\lambda}) = s_{\lambda'}$, where $\lambda'$ denotes the transpose partition of $\lambda$: $(\lambda')_{i} = \#\{ j \in [\ell(\lambda)] \;|\; \lambda_{j} \ge i\}$ for $1 \le i \le \lambda_{1}$.

\subsubsection{Combinatorial Hopf algebras}
\label{sec:CHAs}

This section will give an abridged description of the a framework for classifying Hopf algebra homomorphisms to $\QSym$ established in the paper~\cite{AgBerSot}.  The original result  also includes explicit formulas for all such maps, which is omitted from this paper as the relevant maps are already known.

A \emph{combinatorial Hopf algebra (CHA)} is a pair $(H, \zeta)$ where $H$ is a Hopf algebra and $\zeta: H \shortto \CC$ is an algebra homomorphism, which will be called a \emph{linear character} in order to avoid confusion with group characters.  An important example of a CHA is $\QSym$ with the \emph{first principal specialization}, 
\[
(\QSym, \operatorname{ps}_{1})
\qquad \text{with} \qquad
\begin{array}{rccc}
\operatorname{ps}_{1}\colon & \QSym & \to & \CC \\[0.1em]
& M_{\alpha} & \mapsto & \begin{cases} 1 & \text{if $\ell(\alpha) \le 1$,} \\ 0 & \text{otherwise.} \end{cases}
\end{array}
\]

A \emph{CHA morphism} between combinatorial Hopf algebras $(H, \zeta)$ and $(H', \zeta')$ is a graded Hopf algebra homomorphism $\Psi: H \shortto H'$ for which $\zeta = \zeta' \circ \Psi$.  For example, the inclusion of $\Sym$ into $\QSym$ gives a CHA morphism $(\Sym, \operatorname{ps}_{1}) \shortto (\QSym, \operatorname{ps}_{1})$.

\begin{thm}[\cite{AgBerSot} Theorem 4.1]
\label{thm:universalhopfmap}
Let $(H, \zeta)$ be a combinatorial Hopf algebra.  There is a unique CHA morphism 
\[
\cano
\colon (H, \zeta) \to (\QSym, \operatorname{ps}_{1}).
\]
%
\end{thm}



\subsection{Dyck paths and related objects}
\label{sec:Catalan}
\newcommand{\area}{\operatorname{Area}}
\newcommand{\graph}{\operatorname{Graph}}

The results of this paper build on the combinatorics of Dyck paths, indifference graphs, and Schr\"{o}der paths, each of which are described in this section.

%
A \emph{Dyck path} of size $n \ge 0$ is a lattice path consisting of $2n$ steps east $E = (1, 0)$ and south $S = (0, -1)$ from $(0, 0)$ to $(n, -n)$ which does not go below the main diagonal $y = -x$.  
Let
\[
\DDD = \bigsqcup_{n \ge 0} \DDD_{n} \qquad\text{with}\qquad \DDD_{n} = \{\text{Dyck paths of size $n$}\}.
\]
For example,
\begin{equation}
\label{eq:DyckExample}
\begin{tikzpicture}[scale = 0.45, baseline = 0.45*-1.7cm]
\draw[gray] (-0.2, 0.2) grid (3.2, -3.2);
\draw[dashed, gray] (-0.2, 0.2) -- (3.2, -3.2);
\draw[very thick] (0, 0) -- (2, 0) -- (2, -1) -- (3, -1) -- (3, -3);
\end{tikzpicture} = (EESESS) \in \DDD_{3}.
\end{equation}
It is well known that the size of $\DDD_{n}$ is the $n$th Catalan number, $\frac{1}{n+1} \binom{2n}{n}$; see e.g.~\cite{StaCat}.

An \emph{indifference graph} of size $n \ge 0$ is a simple, unidrected graph $\gamma$ with vertex set $[n] = \{1, \ldots, n\}$ and edge set $E(\gamma)$ satisfying
\[
\text{for each $\{i, l\} \in E(\gamma)$:} \qquad \big\{\{j, k\} \;|\; i \le j < k \le l\big\} \subseteq E(\gamma).
\]
The empty graph on $\emptyset$ is the unique indifference graph of size zero.  
Let
\[
\III\GGG = \bigsqcup_{n \ge 0} \III\GGG_{n}\qquad\text{with}\qquad \III\GGG_{n} = \{\text{indifference graphs on $[n]$}\}.
\]
For example,
\[
\gamma = 
\begin{tikzpicture}[scale = 0.5, baseline = 0.5*-0.3cm]
\draw[fill] (0, 0) circle (2pt) node[inner sep = 1pt] (1) {};
\draw[fill] (1, 0) circle (2pt) node[inner sep = 1pt] (2) {};
\draw[fill] (2, 0) circle (2pt) node[inner sep = 1pt] (3) {};
\draw[fill] (3, 0) circle (2pt) node[inner sep = 1pt] (4) {};
\draw[below] (1) node {$\scriptstyle 1$};
\draw[below] (2) node {$\scriptstyle 2$};
\draw[below] (3) node {$\scriptstyle 3$};
\draw[below] (4) node {$\scriptstyle 4$};
\draw (1) -- (2);
\draw (2) -- (3);
\draw (1) to[out = 45, in = 135] (3);
\draw (3) -- (4);
\end{tikzpicture}
\in \III\GGG_{4}
\qquad\text{but}\qquad
\sigma = 
\begin{tikzpicture}[scale = 0.5, baseline = 0.5*-0.3cm]
\draw[fill] (0, 0) circle (2pt) node[inner sep = 1pt] (1) {};
\draw[fill] (1, 0) circle (2pt) node[inner sep = 1pt] (2) {};
\draw[fill] (2, 0) circle (2pt) node[inner sep = 1pt] (3) {};
\draw[fill] (3, 0) circle (2pt) node[inner sep = 1pt] (4) {};
\draw[below] (1) node {$\scriptstyle 1$};
\draw[below] (2) node {$\scriptstyle 2$};
\draw[below] (3) node {$\scriptstyle 3$};
\draw[below] (4) node {$\scriptstyle 4$};
\draw (1) -- (2);
\draw (2) -- (3);
\draw (1) to[out = 45, in = 135] (3);
\draw (1) to[out = 55, in = 125] (4);
\end{tikzpicture}
\notin \III\GGG_{4},
\]
as $\{1, 4\} \in E(\sigma)$ but $\{3, 4\} \notin E(\sigma)$.

There is a size-preserving bijection between Dyck paths and indifference graphs.  Label the unit squares above $y = -x$ in the fourth quadrant of $\ZZ \times \ZZ$ by edges so the square centered at $(\frac{1}{2}-j, i - \frac{1}{2})$ is labeled by $\{i, j\}$; for example
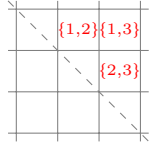
\begin{equation*}
\begin{tikzpicture}[scale = 0.55, baseline = 0.55*-1.7cm]
\draw[gray] (-0.2, 0.2) grid (3.2, -3.2);
\draw[dashed, gray] (-0.2, 0.2) -- (3.2, -3.2);
\draw[fill, red] (1.5, -0.5) node{$\scriptscriptstyle \{1, 2\}$};
\draw[fill, red] (2.5, -0.5) node{$\scriptscriptstyle \{1, 3\}$};
\draw[fill, red] (2.5, -1.5) node {$\scriptscriptstyle \{2, 3\}$};
\end{tikzpicture}
\end{equation*}
shows the first three of these unit squares with their labels.  For any Dyck path $\pi$, let
\[
\area(\pi) = \big\{ \{i, j\} \;\big|\; \text{the unit square $\{i, j\}$ is below $\pi$} \big\}
\]
and if $\pi$ has size $n$, define the \emph{graph of $\pi $} to be 
\[
\graph(\pi) = \big([n], \area(\pi)\big).
\]
For example, taking the Dyck path in~\eqref{eq:DyckExample}, 
\begin{equation}
\label{eq:DyckGraphExample}
\area\left( 
\begin{tikzpicture}[scale = 0.45, baseline = 0.45*-1.7cm]
\draw[gray] (-0.2, 0.2) grid (3.2, -3.2);
\draw[dashed, gray] (-0.2, 0.2) -- (3.2, -3.2);
\draw[very thick] (0, 0) -- (2, 0) -- (2, -1) -- (3, -1) -- (3, -3);
\end{tikzpicture} 
\right) = \big\{\{1, 2\}, \{2, 3\}\big\}
\qquad \text{and} \qquad
\graph\left( 
\begin{tikzpicture}[scale = 0.45, baseline = 0.45*-1.7cm]
\draw[gray] (-0.2, 0.2) grid (3.2, -3.2);
\draw[dashed, gray] (-0.2, 0.2) -- (3.2, -3.2);
\draw[very thick] (0, 0) -- (2, 0) -- (2, -1) -- (3, -1) -- (3, -3);
\end{tikzpicture} 
\right) = 
\begin{tikzpicture}[scale = 0.5, baseline = 0.5*-0.3cm]
\draw[fill] (0, 0) circle (2pt) node[inner sep = 1pt] (1) {};
\draw[fill] (1, 0) circle (2pt) node[inner sep = 1pt] (2) {};
\draw[fill] (2, 0) circle (2pt) node[inner sep = 1pt] (3) {};
\draw[below] (1) node {$\scriptstyle 1$};
\draw[below] (2) node {$\scriptstyle 2$};
\draw[below] (3) node {$\scriptstyle 3$};
\draw (1) -- (2);
\draw (2) -- (3);
\end{tikzpicture}.
\end{equation}

\begin{prop}[{\cite[Solution 187]{StaCat}}]
For $n \ge 0$, the map $\pi \shortmapsto \graph(\pi)$ is a bijection from $\DDD_{n}$ to $\III\GGG_{n}$.
\end{prop}


A common generalization of Dyck paths will appear in Sections~\ref{sec:mainresult2} and~\ref{sec:positivity}.  
A \emph{Schr\"{o}der path} of size $n \ge 0$ is a lattice path from $(0, 0)$ to $(n, -n)$ consisting of steps $E$, $S$, and $D = (1, -1)$ that never goes below the main diagonal.  Thus, every Dyck path is a Schr\"{o}der path, but there are more Schr\"{o}der paths, for example
\begin{equation}
\label{eq:SchroderExample}
\begin{tikzpicture}[scale = 0.45, baseline = 0.45*-1.7cm]
\draw[gray] (-0.2, 0.2) grid (3.2, -3.2);
\draw[dashed, gray] (-0.2, 0.2) -- (3.2, -3.2);
\draw[very thick] (0, 0) -- (2, 0)  -- (3, -1) -- (3, -3);
\end{tikzpicture}
= (EEDSS).
\end{equation}


Say that a Schr\"{o}der path $\sigma$ is \emph{tall} if $\sigma$ has no $D$ steps along the main diagonal.  Let
\[
\TTT\SSS = \bigsqcup_{n \ge 0} \TTT\SSS_{n} \qquad\text{with}\qquad \TTT\SSS_{n} = \{\text{tall Schr\"{o}der paths of size $n$}\}.
\]
The Sch\"{o}der path in~\eqref{eq:SchroderExample} above is tall, as is any Dyck path, taken as a Schr\"{o}der path.  
The number of tall Schr\"{o}der paths by size is given by the small Schr\"{o}der numbers,~\cite[A001003]{OEIS}.  


Finally, for any tall Schr\"{o}der path $\sigma \in \TTT\SSS$, define
\[
\area(\sigma) = \big\{ \{i, j\} \;\big|\; \text{the unit square $\{i, j\}$ is completely below $\sigma$} \big\}
\]
and
\[
\diag(\sigma) = \big\{ \{i, j\} \;|\; \text{$\sigma$ has a diagonal step through the unit square $\{i, j\}$} \big\},
\]
so that taking $\sigma$ as in Equation~\eqref{eq:SchroderExample} gives  $\area(\sigma) = \big\{\{1, 2\}, \{2, 3\}\big\}$ and $\diag(\sigma) = \big\{\{1, 3\}\big\}$.

\subsection{Supercharacter Theory}
\label{sec:sct}

Let $G$ be a finite group, let $\operatorname{Irr}(G)$ denote the irreducible complex characters of $G$, and let $\cf(G)$ denote the space of complex-valued class functions on $G$.  A \emph{supercharacter theory} $(\mathtt{Cl}, \mathtt{Ch})$ of $G$ comprises a set partition $\mathtt{Cl}$ of $G$ and a basis of orthogonal characters $\mathtt{Ch}$ for the space
\[
\scf(G) = \{ \phi: G \shortto \CC \;|\; \text{$\phi$ is constant on each part of $\mathtt{Cl}$}\},
\]
such that $\scf(G)$ contains the regular character $\operatorname{reg}_{G}$ of $G$.  
In this definition, orthogonality is determined by the usual inner product $\langle \cdot, \cdot \rangle: \cf(G) \otimes \cf(G) \shortto \CC$.  

The elements of $\mathtt{Cl}$ and $\mathtt{Ch}$ are respectively called \emph{superclasses} and \emph{supercharacters}.
Every group has at least one supercharacter theory, with superclasses given by conjugacy classes and supercharacters given by irreducible characters, and in this case $\scf(G) = \cf(G)$.

Each supercharacter theory of $G$ comes with two canonical bases: the supercharacters in $\mathtt{Ch}$ and the set of \emph{superclass identifier functions}
\begin{equation}
\label{eq:classindicator}
\{\delta_{K} \;|\; \text{$K \in \mathtt{Cl}$}\}
\qquad\text{with}\qquad
\delta_{K}(g) = \begin{cases} 1 & \text{if $g \in K$,} \\ 0 & \text{otherwise} \end{cases}
\end{equation}
Up to scaling these bases are self-dual with respect to $\langle \cdot, \cdot \rangle$, and I will identify
\begin{equation}
\label{eq:cfdual}
\scf(G)^{\ast} = \{\ket{\chi} \;|\; \chi \in \scf(G)\}.
\qquad\text{with}\qquad 
\begin{array}{rcl}
\ket{\chi}: \scf(G) & \to & \CC \\
\psi & \mapsto & \langle \psi, \chi \rangle
\end{array}.
\end{equation}

The remainder of the section describes a particular collection of supercharacter theories originating in the work of Aliniaeifard and Thiem~\cite{AlTh20}.  Fix a prime power $q$, let $\FF_{q}$ denote the field with $q$ elements, and let $\GL_{n} = \GL_{n}(\FF_{q})$.  The \emph{unipotent upper triangular group} is the subgroup
\[
\UT_{n} = \{ g \in \GL_{n} \;|\; \text{$(g - 1_{n})_{i, j} \neq 0$ only if $i < j$} \}
\]
where $1_{n}$ denotes the $n \times n$ identity matrix.  This group has a family of normal subgroups---called \emph{normal pattern subgroups}---indexed by indifference graphs~\cite[Lemma~4.1]{Mar}: for $\gamma \in \III\GGG_{n}$, let
\[
\UT_{\gamma} = \{ g \in \UT_{n} \;|\; \text{$g_{i, j} = 0$ if $\{i, j\} \in E(\gamma)$} \} 
\]
where $E(\gamma)$ denotes the edge set of $\gamma$.  If $\pi \in \DDD_{n}$ is the Dyck path for which $\gamma = \graph(\pi)$, $\UT_{\gamma}$ can be visualized in terms of $\pi$: $\UT_{\gamma}$ is the subset of elements of $\UT_{n}$ with nonzero entries occurring only on the diagonal or above the path $\pi$.  For example using the graph and Dyck path from Equation~\eqref{eq:DyckGraphExample},
\[
\UT_{\hspace{-0.25em} \begin{tikzpicture}[scale = 0.35, baseline = 0.35*-0.3cm]
\draw[fill] (0, 0) circle (2pt) node[inner sep = 1pt] (1) {};
\draw[fill] (1, 0) circle (2pt) node[inner sep = 1pt] (2) {};
\draw[fill] (2, 0) circle (2pt) node[inner sep = 1pt] (3) {};
\draw[below] (1) node {$\scriptscriptstyle 1$};
\draw[below] (2) node {$\scriptscriptstyle 2$};
\draw[below] (3) node {$\scriptscriptstyle 3$};
\draw (1) -- (2);
\draw (2) -- (3);
\end{tikzpicture}\hspace{-0.25em} }
= 
\begin{tikzpicture}[scale = 0.45, baseline = 0.45*-1.7cm]
\path[fill = gray!35] (2, 0) -- (3, 0) -- (3, -1) -- (2, -1) -- cycle;
\draw[gray] (0, 0) grid (3, -3);
\draw[gray, thick] (0, 0) -- (3, 0) -- (3, -3) -- (0, -3) -- cycle;
\draw[very thick] (0, 0) -- (2, 0) -- (2, -1) -- (3, -1) -- (3, -3);
\foreach \x in {1, 2, 3}{\draw (\x - 0.5, -\x + 0.5) node {$1$};}
\foreach \x in {1}{\foreach \y in {2, 3}{\draw (\x - 0.5, -\y + 0.5) node {$0$};}}
\foreach \x in {2}{\foreach \y in {3}{\draw (\x - 0.5, -\y + 0.5) node {$0$};}}
\foreach \x in {2}{\foreach \y in {1}{\draw (\x - 0.5, -\y + 0.5) node {$0$};}}
\foreach \x in {3}{\foreach \y in {1}{\draw (\x - 0.5, -\y + 0.5) node {$\ast$};}}
\foreach \x in {3}{\foreach \y in {2}{\draw (\x - 0.5, -\y + 0.5) node {$0$};}}
\begin{scope}[xshift = -0.05cm]
\draw (0.2, 0.085) -- (-0.15, 0.085);
\draw[very thick] (-0.1, 0.1) -- (-0.1, -3.1);
\draw (-0.15, -3.085) -- (0.2, -3.085);
\end{scope}
\begin{scope}[xshift = 0.05cm]
\draw (2.8, 0.085) -- (3.15, 0.085);
\draw[very thick] (3.1, 0.1) -- (3.1, -3.1);
\draw (3.15, -3.085) -- (2.8, -3.085);
\end{scope}
\end{tikzpicture}.
\]

The set of normal pattern subgroups of $\UT_{n}$ form a lattice under containment.  This order is dual to the spanning subgraph relation on $\III\GGG_{n}$, in that the containment $\UT_{\gamma} \subseteq \UT_{\sigma}$ holds if and only if $\sigma$ is a spanning subgraph of $\gamma$.  The maximal normal pattern subgroup is $\UT_{n}$, corresponding to the edgeless graph $([n], \emptyset)$, and $|\UT_{n}: \UT_{\gamma}| = q^{|E(\gamma)|}$ for all $\gamma \in \III\GGG_{n}$.

The lattice structure on normal pattern subgroups partitions the set $\UT_{n}$ into parts
\begin{align*}
\UT_{\gamma}^{\circ} &= \{g \in \UT_{\gamma} \;|\; \text{$g \notin \UT_{\sigma}$ for any $\sigma \supsetneq \gamma$}\} 
\end{align*}
for each $\gamma \in \III\GGG_{n}$.  Similarly, $\III\GGG_{n}$ indexes the parts of a partition of the set of irreducible characters $\operatorname{Irr}(\UT_{n})$ of $\UT_{n}$: let
\[
\widehat{\UT}_{\gamma}^{\circ} = \{\psi \in \operatorname{Irr}(\UT_{n}) \;|\; \text{$\UT_{\gamma} \subseteq \ker(\psi)$ and $\UT_{\sigma} \not \subseteq \ker(\psi)$ for each $\sigma \supsetneq \gamma$} \}.
\]
for each $\gamma \in \III\GGG_{n}$, and further define
\[
\chi^{\gamma} = \sum_{\psi \in \widehat{\UT}_{\gamma}^{\circ}} \psi(1) \psi.
\]

\begin{prop}[{\cite[Section~3.2]{AlTh20}}]
With
\[
\mathtt{Ch} = \{ \UT^{\circ}_{\gamma} \;|\; \gamma \in \III\GGG_{n} \}
\qquad\text{and}\qquad
\mathtt{Cl} = \{ \chi^{\gamma} \;|\; \gamma \in \III\GGG_{n} \},
\]
the pair $(\mathtt{Cl}, \mathtt{Ch})$ is a supercharacter theory of $\UT_{n}$.
\end{prop}

For the remainder of the paper, write $\delta_{\gamma} = \delta_{\UT^{\circ}_{\gamma}}$ for the superclass identifier functions in this supercharacter theory.  In addition to these functions and the supercharacters, the space $\scf(\UT_{n})$ has two interesting bases: $\{\permind_{\gamma} \;|\; \gamma \in \III\GGG_{n}\}$ and $\{\permchar^{\gamma} \;|\; \gamma \in \III\GGG_{n}\}$, with
\[
\permind_{\gamma} = \sum_{\sigma \supseteq \gamma} \delta_{\sigma} 
\qquad\text{and}\qquad
\permchar^{\gamma} = \sum_{\sigma \subseteq \gamma} \chi^{\sigma}.
\]
Remarkably, if $\mathbbm{1} \in \cf(\UT_{\gamma})$ denotes the character of the trivial representation then
\begin{equation}
\label{eq:permtoind}
\permchar^{\gamma} = \ind^{\UT_{n}}_{\UT_{\gamma}}(\mathbbm{1}) = q^{|E(\gamma)|} \permind_{\gamma},
\end{equation}
the character of the $\UT_{n}$-module $\CC[\UT_{n}/\UT_{\gamma}]$.

\subsection{Homomorphisms between Hopf algebras of class functions}
\label{sec:repHopf}

In~\cite[\nopp III]{Zel}, Zelevinsky defines a graded connected Hopf algebra on the space
\[
\cf(\GL_{\bullet}) = \bigoplus_{n \ge 0} \cf(\GL_{n}),
\]
with structure maps coming from the parabolic induction and restriction functors.  
The paper~\cite{Gaga} defines a similar Hopf structure on the spaces
\[
\scf(\UT_{\bullet}) = \bigoplus_{n \ge 0} \scf(\UT_{n}),
\qquad\text{and}\qquad
\cf(\UT_{\bullet}) = \bigoplus_{n \ge 0} \cf(\UT_{n}),
\]
in which $\scf(\UT_{n})$ is the subspace of class functions defined in Section~\ref{sec:sct}, with $\scf(\UT_{\bullet})$ a sub-Hopf algebra of $\cf(\UT_{\bullet})$.  This section will describe several homomorphisms involving these Hopf algebras.

In~\cite[Section 6]{GP16}, Guay-Paquet defines a $\CC[t]$-Hopf algebra on the free $\CC[t]$-module $\CC[t][\III\GGG]$, and specializing $t \shortmapsto q^{-1}$ gives a Hopf algebra over $\CC$; see~\cite[Section~7]{Gaga}.  
Recall the basis $\{\permind_{\gamma} \;|\; \gamma \in \III\GGG\}$ of $\scf(\UT_{\bullet})$ defined in Section~\ref{sec:sct}.

\begin{thm}[{\cite[Corollary~7.2]{Gaga}}]
\label{thm:scfNG}
The map $\gamma \shortmapsto \permind_{\gamma}$ is an isomorphism from Guay-Paquet's specialized Hopf algebra to $\scf(\UT_{\bullet})$.
\end{thm}

A second map comes from the induction functors $\ind_{\UT_{n}}^{\GL_{n}} : \cf(\UT_{n}) \shortto \cf(\GL_{n})$:  let
\begin{equation}
\label{eq:indUTGL}
\ind_{\UT}^{\GL} = \bigoplus_{n \ge 0} \ind_{\UT_{n}}^{\GL_{n}}: \cf(\UT_{\bullet}) \to \cf(\GL_{\bullet}).
\end{equation}

\begin{thm}[{\cite[Theorem~6.1]{Gaga}}]
The map $\ind^{\GL}_{\UT}$ is a Hopf algebra homomorphism.
\end{thm}

The homomorphism $\ind_{\UT}^{\GL}$ also induces a linear map on dual spaces.  Using the identification in~\eqref{eq:classindicator}, the dual of the direct sum $\cf(\GL_{\bullet})$ becomes a product
\[
\cf(\GL_{\bullet})^{\ast}  = \prod_{n \ge 0} \cf(\GL_{n})^{\ast} = \Big\{ \big( \ket{\chi_{n}}  \big)_{n \ge 0} \;\Big|\; \chi_{n} \in \cf(\GL_{n}) \Big\}.
\]
Making the analogous identification for $\cf(\UT_{\bullet})^{\ast}$ and $\scf(\UT_{\bullet})^{\ast}$, Frobenius reciprocity gives that, 
\[
\big( \ket{\chi_{n}}  \big)_{n \ge 0}  \circ \ind_{\UT}^{\GL} = \big( \ket{\res^{\GL_{n}}_{\UT_{n}}(\chi_{n})}  \big)_{n \ge 0}.
\]
Furthermore, if $\res^{\GL_{n}}_{\UT_{n}}(\chi_{n}) \in \scf(\UT_{\bullet})$ for each $n \ge 0$, the same equation applies when considering each side as an element of $\scf(\UT_{\bullet})^{\ast}$.

\section{The chromatic quasisymmetric function as a $\GL_{n}$ character}
\label{sec:mainresult1}

This section will state and prove Theorem~\ref{thm:CQSbijection}, following some initial context.  Recall the Hopf algebras $\scf(\UT_{\bullet})$ and $\cf(\UT_{\bullet})$ from Section~\ref{sec:repHopf}.  
The Hopf algebra of \emph{$\GL$-class functions with unipotent support} is the image
\[
\cfunsupp(\GL_{\bullet}) = \ind^{\GL}_{\UT}(\cf(\UT_{\bullet})) \subseteq \cf(\GL_{\bullet}).
\]
Zelevinsky~\cite{Zel} has defined a Hopf algebra isomorphism $\canoUS: \cfunsupp(\GL_{\bullet}) \shortto \Sym$ which will be used in the theorem; see Section~\ref{sec:cfunsupp}.  
Finally, for each indifference graph $\gamma$, recall the subgroup $\UT_{\gamma}$ defined in Section~\ref{sec:sct}, and let $X_{\gamma}(\mathbf{x}; t)$ denote the chromatic quasisymmetric function of $\gamma$ in an indeterminate `$t$', which will be formally defined in Section~\ref{sec:CQS}.


\begin{thm}
\label{thm:CQSbijection}
For $n \ge 0$ and $\gamma \in \III\GGG_{n}$,
\[
\ind^{\GL_{n}}_{\UT_{\gamma}}(\mathbbm{1})
 = (q-1)^{n} \, \canoUS^{-1}\big(X_{\gamma}(\mathbf{x}; q)\big).
\]
\end{thm}

I will describe briefly how the results in this section prove Theorem~\ref{thm:CQSbijection}.  
Define a Hopf algebra homomorphism $\canoCQS\colon \scf(\UT_{\bullet}) \shortto \QSym$ as the composite map in the diagram
\begin{equation}
\label{eq:CQSdiagram}
\begin{tikzpicture}[baseline = -1.2cm]
\node at (0, 0) (UT) {$ \scf(\UT_{\bullet})$};
\node at (0, -2) (GL) {$\cfunsupp(\GL_{\bullet})$};
\node at (4, -2) (Sym) {$\Sym$};
\node at (8, -2) (QSym) {$\QSym$};
\draw[thick, -latex] (UT) -- node[left] {$\ind^{\GL}_{\UT}$} (GL); 
\draw[thick, -latex, dashed] (UT) -- node[above] {$\canoCQS$} (QSym); 
\draw[thick, -latex] (GL) -- node[below] {$\canoUS$} node[above]  {$\cong$} (Sym); 
\draw[thick,  right hook-latex] (Sym) -- node[below] {inclusion} (QSym); 
\end{tikzpicture}
\end{equation}
of Hopf algebra homomorphisms.  By the transitivity of induction, the theorem is equivalent to computing the image of the character $\permchar^{\gamma} = \ind^{\UT_{n}}_{\UT_{\gamma}}(\mathbbm{1}) \in \scf(\UT_{\bullet})$ under $\canoCQS$.  

Recalling the theory of combinatorial Hopf algebras from Section~\ref{sec:CHAs}, there is a unique combinatorial Hopf algebra structure on $\scf(\UT_{\bullet})$ for which $\canoCQS$ is a CHA morphism to $(\QSym, \operatorname{ps}_{1})$, and this structure is given by a linear character of the Hopf algebra $\scf(\UT_{\bullet})$.  
Theorem~\ref{thm:canoCQS} computes this linear character, and Proposition~\ref{prop:CQSchar} shows that it is essentially the same as one defined by Guay-Paquet in~\cite{GP16}.  
This leads to a formula for $\canoCQS$ on the basis $\{ \permind_{\gamma} \;|\; \gamma \in \III\GGG\}$ of $\scf(\UT_{\bullet})$ from Section~\ref{sec:sct}, stated formally in Corollary~\ref{cor:GPidentity}:
\begin{equation}
\label{eq:GPidentity}
\canoCQS(\permind_{\gamma}) = (q-1)^{n}X_{\gamma}(\mathbf{x}; q^{-1}) \qquad \text{for $\gamma \in \III\GGG_{n}$}.
\end{equation}
From here, the theorem follows from an identity of Shareshian--Wachs~\cite{ShWa}.  Recalling from Section~\ref{sec:sct} that $\permchar^{\gamma} = q^{|E(\gamma)|} \permind_{\gamma}$,~\cite[Proposition 2.6]{ShWa} reformulats Equation~\eqref{eq:GPidentity} as
\[
\canoCQS(\permchar^{\gamma}) 
= 
(q-1)^{n} q^{|E(\gamma)|}  X_{\gamma}(\mathbf{x}; q^{-1}) 
= 
(q-1)^{n} X_{\gamma}(\mathbf{x}; q).
\]


The results used in the proof are given in the remainder of this section, which comprises two subsections.  
Section~\ref{sec:cfunsupp} describes the linear characters of the Hopf algebras $\cfunsupp(\GL_{\bullet})$ and $\scf(\UT_{\bullet})$ needed to make Diagram~\eqref{eq:CQSdiagram} a diagram of combiantorial Hopf algebras.  
Then, Section~\ref{sec:CQS} uses results from~\cite{GP16} and Section~\ref{sec:prelims} to describe the chromatic quasisymmetric function as the image of a CHA morphism from $\scf(\UT_{\bullet})$ and subsequently shows that up to a power of $(q-1)$ this map coincides with $\canoCQS$.

\subsection{Factoring $\canoCQS$ through $\cfunsupp(\GL_{\bullet})$}
\label{sec:cfunsupp}


Say that an element $g \in \GL_{n}$ is \emph{unipotent} if $g$ is conjugate to an element of $\UT_{n}$, so that by definition the sub-Hopf algebra $\cfunsupp(\GL_{\bullet})$ from Section~\ref{sec:repHopf} is 
\[
\cfunsupp(\GL_{\bullet}) = \bigoplus_{n \ge 0} \{\psi \in \cf(\GL_{n}) \;|\; \text{$\psi(h) = 0$ for $h \in \GL_{n}$ not unipotent}\}.
\]
The conjugacy classes of unipotent elements in $\GL_{n}$ are indexed by $\PPP_{n}$: the partition $\lambda = (\lambda_{1}, \lambda_{2}, \ldots, \lambda_{\ell})$ corresponds to the conjugacy class $O_{\lambda}$ of the Jordan matrix
\[
J_{\lambda} = \begin{bmatrix} J_{\lambda_{1}} & 0 & \cdots & 0 \\ 0 & J_{\lambda_{2}} & \cdots & 0 \\ \vdots & \vdots & \ddots & \vdots \\ 0 & 0 & \cdots & J_{\lambda_{\ell}} \end{bmatrix} 
\qquad \text{with} \qquad 
J_{k} = \begin{bmatrix} 
1 & 1 & 0 & \cdots & 0 \\ 
0 & 1 & 1 & \cdots & 0 \\ 
\vdots & \vdots & \vdots & \ddots &  \vdots \\ 
0 & 0 & 0 & \cdots & 1\end{bmatrix}. 
\]
Thus, $\cfunsupp(\GL_{\bullet})$ has a $\PPP$-indexed basis of identifier functions $\delta_{\lambda} = \delta_{O_{\lambda}}$ for unipotent conjugacy classes,
\[
\cfunsupp(\GL_{\bullet}) = \CC\spanning\{ \delta_{\lambda} \;|\; \lambda \in \PPP \}.
\]
Zelevinsky~\cite[\nopp 10.13]{Zel} (see also \cite[\nopp IV.4.1]{Mac}) constructs a graded Hopf algebra isomorphism
\begin{equation}
\label{eq:canoUCmap}
\begin{array}{rcl}
\canoUS\colon \cfunsupp(\GL_{\bullet}) & \to & \Sym \\[0.5em]
\delta_{\lambda} & \mapsto & \widetilde{P_{\lambda}}(\mathbf{x}; q) = q^{-n(\lambda)} P_{\lambda}(\mathbf{x}; q^{-1})
\end{array}
\end{equation}
where $P_{\lambda}(\mathbf{x}; t) \in \Sym[t]$ is the Hall--Littlewood polynomial~\cite[\nopp III.2]{Mac} and $n(\lambda) = \sum_{i = 1}^{\lambda_{1}} \binom{\lambda'_{i}}{2}$.


In the framework of Theorem~\ref{thm:universalhopfmap}, the isomorphism $\canoUS$ is equivalent to a linear character of $\cfunsupp(\GL_{\bullet})$.   
This datum was also determined by Zelevinsky in~\cite{Zel}.  
The \emph{regular} unipotent elements of $\GL_{n}$ are the members of the conjugacy class $O_{(n)}$.  Using the notation of Section~\ref{sec:repHopf}, define a linear functional
\[
\ket{\zetaUS} = \big(\ket{\delta^{\ast}_{\USind}}\big)_{n \ge 0} \in \cf(\GL_{\bullet})^{\ast}\qquad \text{with} \qquad  \ket{\delta^{\ast}_{\USind}} = \frac{\ket{\delta_{\USind}}}{\langle \delta_{\USind}, \delta_{\USind} \rangle},
\]
so that for $\psi \in \cf(\GL_{n})$, the value of $\ket{\zetaUS}(\psi)$ is the the value of $\psi$ at any regular unipotent element, $\psi(J_{(n)})$.  By embedding $\cfunsupp(\GL_{\bullet})$ into $\cf(\GL_{\bullet})$, $\ket{\zetaUS}$ is also a linear functional on $\cfunsupp(\GL_{\bullet})$.

\begin{prop}[{\cite[\nopp 10.8]{Zel}}]
\label{prop:cfunsupp}
The map $\ket{\zetaUS}$ is a linear character of the Hopf algebra $\cfunsupp(\GL_{\bullet})$, and moreover $\canoUS$ is the unique CHA morphism $(\cfunsupp(\GL_{\bullet}), \ket{\zetaUS}) \shortto (\Sym, \operatorname{ps}_{1})$.
\end{prop}
%

Now consider the Hopf algebra $\scf(\UT_{\bullet})$.  
Recall that $([n], \emptyset)$ is the minimal indifference graph on $n$ vertices and define a linear functional
\[
\ket{\zetaCQSbullet} = \bigg((q-1)^{n} \ket {\delta_{\UTnInd}^{\ast} }\bigg)_{n \ge 0} \hspace{-0.5em} \in \scf(\UT_{\bullet})^{\ast} \qquad \text{with} \qquad  \ket{\delta^{\ast}_{\UTnInd}} = \frac{\ket{\delta_{\UTnInd}}}{\langle \delta_{\UTnInd}, \delta_{\UTnInd} \rangle}.
\]

\begin{rem}
There is an unfortunate conincidence of notation between the class functions $\delta^{\ast}_{\USind}$ and $\delta_{\UTnInd}^{\ast}$, and care should be taken to distinguish between the two: up to normalization $\delta^{\ast}_{\USind}$ is the $\GL_{n}$-class function which identifies the conjugacy class $O_{(n)}$ of regular unipotent elements, and $\delta_{\UTnInd}^{\ast}$ is the $\UT_{n}$-class function which identifies the superclass
\[
\UT_{\UTnInd}^{\circ} = \{X \in \UT_{n} \;|\; \text{$X_{i, i+1} \neq 0$ for $1 \le i < n$}\}.
\]
However, the two are closely related, as described in the proof of Theorem~\ref{thm:canoCQS} below.
\end{rem}

\begin{thm}
\label{thm:canoCQS}
The linear functional $\ket{\zetaCQSbullet}$ is a linear character of $\scf(\UT_{\bullet})$.  Moreover,
\[
\ket{\zetaCQSbullet} = \ket{ \delta_{(\bullet)}^{\ast}} \circ \ind^{\GL}_{\UT},
\]
so $\ind^{\GL}_{\UT}$ is a CHA morphism
\[
\big(\scf(\UT_{\bullet}), \ket{\zetaCQSbullet} \big) \xrightarrow{\;\; \ind^{\GL}_{\UT} \;\;} \big( \cf(\GL_{\bullet}), \ket{ \delta_{(\bullet)}^{\ast}} \big).
\]
\end{thm}
\begin{proof}
The first and third assertions follow from the second.  The proof of the second assertion will make use of the fact that the superclass $\UT_{\UTnInd}^{\circ}$ is also the set of all regular unipotent elements in $\UT_{n}$, so that $\delta_{\UTnInd} = \res^{\GL_{n}}_{\UT_{n}} ( \delta_{(n)} )$.

For $\gamma \in \III\GGG_{n}$, Frobenius reciprocity (as described in Section~\ref{sec:repHopf}) gives
\[
\ket{ \delta_{(\bullet)}^{\ast}} \circ \ind^{\GL}_{\UT} (\permind_{\gamma}) 
= \frac{\ket{\res^{\GL_{n}}_{\UT_{n}} ( \delta_{(n)} )}(\permind_{\gamma})}{\langle \delta_{(n)}, \delta_{(n)} \rangle} 
= \frac{\langle \permind_{\gamma},  \delta_{\UTnInd} \rangle}{\langle \delta_{(n)}, \delta_{(n)} \rangle} = \begin{cases} \frac{\langle \delta_{\UTnInd}, \delta_{\UTnInd} \rangle}{\langle \delta_{(n)}, \delta_{(n)} \rangle} & \text{if $\gamma = \UTnInd$,} \\ 0 & \text{otherwise,} \end{cases}
\]
with the last equation following from the definition of $\permind_{\gamma}$, the minimality of $\UTnInd$, and the orthogonality of the superclass identifiers, see Section~\ref{sec:sct}.  
Direct computation then gives that
\[
\frac{\langle \delta_{\UTnInd}, \delta_{\UTnInd} \rangle}{\langle \delta_{(n)}, \delta_{(n)} \rangle} = \frac{|\GL_{n}|}{|O_{(n)}|} \frac{| \UT_{\UTnInd}^{\circ} |}{|\UT_{n}|} = (q-1)^{n},
\]
where the final equality comes from the order formulas 
\[
O_{(n)} = \frac{|\GL_{n}|}{q^{n-1}(q-1)}
\qquad\text{and}\qquad
\UT_{\UTnInd}^{\circ} = (q-1)^{n-1} \frac{|\UT_{n}|}{q^{n-1}}. \qedhere
\]
\end{proof}

Now recall the map $\canoCQS$ defined in Diagram~\ref{eq:CQSdiagram}.  Theorem~\ref{thm:canoCQS} and Proposition~\ref{prop:cfunsupp} give the following.

\begin{cor}
The map $\canoCQS$ is the unique CHA morphism
\[
\canoCQS: \big(\scf(\UT_{\bullet}), \ket{\zetaCQSbullet} \big) \to \big( \QSym, \operatorname{ps}_{1} \big).
\]
\end{cor}

\begin{rem}
Theorem~\ref{thm:canoCQS} actually establishes the stronger result that $\ket{\zetaCQSbullet}$ is a linear character of $\cf(\UT_{\bullet})$, and that we may extend the domain of the CHA morphisms $\ind^{\GL}_{\UT}$ and $\canoCQS$ to the combinatorial Hopf algebra $\big(\cf(\UT_{\bullet}), \ket{\zetaCQSbullet} \big)$.  
While this level of generality is unnecessary for the scope of this work, it may be of general interest.
\end{rem}

\subsection{The chromatic quasisymmetric function}
\label{sec:CQS}

This section defines the chromatic quasisymmetric function of a graph and describes how it can be realized as the image of a character of $\GL_{n}(\FF_{q})$ under a particular a CHA morphism, leading to a proof of Theorem~\ref{thm:CQSbijection}.

Let $\gamma$ be a simple, undirected graph with vertex set $[n]$ and edge set $E(\gamma)$.  A \emph{coloring} of $\gamma$ is a function $\kappa: [n] \shortto \ZZ_{>0}$.  A coloring $\kappa$ of $\gamma$ is \emph{proper} if $\kappa(i) \neq \kappa(j)$ for all $\{i, j\} \in E(\gamma)$.  The \emph{$\gamma$-ascent number} of a coloring $\kappa$ is
\begin{equation}
\label{eq:ascstat}
\asc_{\gamma}(\kappa) = \big|\big\{\{i, j\} \in E(\gamma) \;|\; \text{$i < j$ and $\kappa(i) < \kappa(j)$}\big\}\big|.
\end{equation}
For example, if $\kappa: [5] \shortto \ZZ_{> 0}$ is given by $\kappa(1) = 2$, $\kappa(2) = 5$, $\kappa(3) =1$, and $\kappa(4) = 5$, then
\[
\asc_{\hspace{-0.25em}\begin{tikzpicture}[scale = 0.35, baseline = 0.35*-0.25cm]
\draw[fill] (0, 0) circle (2pt) node[inner sep = 1pt] (1) {};
\draw[fill] (1, 0) circle (2pt) node[inner sep = 1pt] (2) {};
\draw[fill] (2, 0) circle (2pt) node[inner sep = 1pt] (3) {};
\draw[fill] (3, 0) circle (2pt) node[inner sep = 1pt] (4) {};
\node[below] at (1) {$\scriptscriptstyle 1$};
\node[below] at (2) {$\scriptscriptstyle 2$};
\node[below] at (3) {$\scriptscriptstyle 3$};
\node[below] at (4) {$\scriptscriptstyle 4$};
\draw (1) to[] (2);
\draw (2) to[]  (3);
\draw (1) to[out = 45, in = 135]  (3);
\draw (3) to[]  (4);
\end{tikzpicture}\hspace{-0.25em}}(\kappa) = \big|\big\{\{1, 2\}, \{3, 4\}\big\}\big| = 2.
\]
In this example, $\kappa$ is a proper coloring of the given graph.

The \emph{chromatic quasisymmetric function} of $\gamma$ is
\[
X_{\gamma}(\mathbf{x}; t) = \sum_{\substack{\kappa: [n] \shortto \ZZ_{>0} \\ \text{proper}}} t^{\asc_{\gamma}(\kappa)} x_{\kappa(1)} x_{\kappa(2)} \dots x_{\kappa(n)} \in \QSym[t],
\]
so that $X_{\gamma}(\mathbf{x}; t)$ is a polynomial in an indeterminate $t$ whose coefficients---by properties of the ascent statistic---are quasisymmetric functions.  For an indifference graph $\gamma \in \III\GGG_{n}$, it is known that these coefficients are elements of $\Sym$~\cite[Theorem 4.5]{ShWa}.  For example,
\[
X_{\hspace{-0.25em}\begin{tikzpicture}[scale = 0.35, baseline = 0.35*-0.35cm]
\draw[fill] (0, 0) circle (2pt) node[inner sep = 1pt] (1) {};
\draw[fill] (1, 0) circle (2pt) node[inner sep = 1pt] (2) {};
\draw[fill] (2, 0) circle (2pt) node[inner sep = 1pt] (3) {};
\node[below] at (1) {$\scriptscriptstyle 1$};
\node[below] at (2) {$\scriptscriptstyle 2$};
\node[below] at (3) {$\scriptscriptstyle 3$};
\draw (1) -- (2);
\draw (2) --  (3);
\end{tikzpicture}\hspace{-0.25em}}(\mathbf{x}; t) = t \; m_{(2, 1)} + (t^{2}+4t+1) \; m_{(1^{3})}.
\]
However, this property is not used, and a novel proof of it follows from Corollary~\ref{cor:GPidentity} below; see Remark~\ref{rem:CQSsymmetry} (R1).

Evaluating the indeterminate $t$ in $X_{\gamma}(\mathbf{x}; t)$ at a complex number gives an actual (quasi-) symmetric function.  For example, $X_{\gamma}(\mathbf{x}; 1)$ is the ordinary chromatic symmetric function of the graph $\gamma$, as defined by Stanley in~\cite{Sta}.  In Theorem~\ref{thm:CQSbijection} the chromatic quasisymmetric functions are evaluated at $q$, the order of the finite field $\FF_{q}$.

In~\cite{GP16}, Guay-Paquet constructs the chromatic quasisymmetric by way of a homomorphism of $\CC[t]$-Hopf algebras.  By evaluating at $t = q^{-1}$ as in Theorem~\ref{thm:scfNG}, this result descends to a Hopf algebra homomorphism $\scf(\UT_{\bullet}) \shortto \QSym$.  Define a linear functional
\[
\begin{array}{rcl}
\zeta_{0}: \scf(\UT_{\bullet}) & \to & \CC \\
\permind_{\gamma} & \mapsto & \begin{cases} 1 & \text{if $\gamma = \UTnInd$,} \\ 0 & \text{otherwise.} \end{cases}
\end{array}
\]

The following theorem is translated from its original context in~\cite{GP16} for the Hopf algebra $\scf(\UT_{\bullet})$ using the Hopf algebra isomorphism in Theorem~\ref{thm:scfNG}.

\begin{thm}[{\cite[Theorem~57]{GP16}}]
\label{thm:GPcano1}
The map $\zeta_{0}$ is a linear character of $\scf(\UT_{\bullet})$, and the unique CHA morphism
\[
\big( \scf(\UT_{\bullet}), \zeta_{0} \big) \to (\QSym, \operatorname{ps}_{1})
\]
is given by
\[
\permind_{\gamma} \mapsto X_{\gamma}(\mathbf{x}; q^{-1}).
\]
\end{thm}


Along with Theorem~\ref{thm:universalhopfmap}, this result is the key to  compute the image of the map $\canoCQS$ defined at the outset of Section~\ref{sec:mainresult1}.  Recall the linear character $\ket{\zetaCQSbullet}$ of the Hopf algebra $\scf(\UT_{\bullet})$ defined in Section~\ref{sec:cfunsupp}.


\begin{prop}
\label{prop:CQSchar}
Let $\gamma$ be an indifference graph of size $n \ge 0$.  Then
\[
\ket{\zetaCQSbullet } (\permind_{\gamma}) 
= \begin{cases} (q-1)^{n} & \text{if $\gamma = \UTnInd$,} \\ 0 & \text{otherwise.} \end{cases}
\]
\end{prop}
\begin{proof}
By definition,  $\permind_{\gamma} = \sum_{\sigma \supseteq \gamma} \delta_{\sigma}$.  Explicit computation then gives
\[
\frac{\langle (q-1)^{n} \delta_{\UTnInd},  \permind_{\gamma} \rangle }{\langle \delta_{\UTnInd}, \delta_{\UTnInd} \rangle} 
=(q-1)^{n} \frac{\sum_{\sigma \supseteq \gamma} \langle \delta_{\UTnInd}, \delta_{\sigma} \rangle}{\langle \delta_{\UTnInd}, \delta_{\UTnInd} \rangle}.
\]
Using the orthogonality of the basis $\{\delta_{\gamma} \;|\; \gamma \in \III\GGG_{n}\}$ and the minimality of $\UTnInd$ under the spanning subgraph order on $\III\GGG_{n}$, the above expression reduces to the desired formula.
\end{proof}

Thus, on homogeneous elements of degree $n$, the characters $\ket{\zetaCQSbullet}$ and $\zeta_{0}$ only differ by a factor of $(q-1)^{n}$.  This leads to the following result, which is a restatement of Equation~\eqref{eq:GPidentity} and accordingly the last step in the proof of Theorem~\ref{thm:CQSbijection}.

\begin{cor}
\label{cor:GPidentity}
Let $\gamma$ be an indifference graph of size $n \ge 0$.  Then
\[
\canoCQS(\permind_{\zeta}) = (q - 1)^{n} X_{\gamma}(\mathbf{x}; q^{-1}).
\]
\end{cor}


\begin{proof}
By comparison with the Hopf algebra homomorphism in Theorem~\ref{thm:GPcano1}, it is clear that the given map is a graded Hopf algebra homomorphism, and further, that
\[
\operatorname{ps}_{1}\big((q - 1)^{n} X_{\gamma}(\mathbf{x}; q^{-1})\big) = (q-1)^{n} \zeta_{0}\big( \permind_{\gamma} \big) = \ket{\zetaCQSbullet } (\permind_{\gamma}).
\]
Thus, the given map is a CHA morphism 
\[
\big(\scf(\UT_{\bullet}), \ket{\zetaCQSbullet} \big) \to \big( \QSym, \operatorname{ps}_{1} \big).
\]
By Theorem~\ref{thm:universalhopfmap}, the above map must be equal to $\canoCQS$.
\end{proof}


\begin{rems}
\label{rem:CQSsymmetry}
\begin{enumerate}[label = (R\arabic*)]
\item As the image of $\canoCQS$ is $\Sym \subseteq \QSym$, Corollary~\ref{cor:GPidentity} gives a novel proof that the coefficients of $X_{\gamma}(\mathbf{x}; t)$ are symmetric functions.

\item Proposition~\ref{prop:CQSchar} also shows that $\zeta_{0} = \big(\ket{\delta_{\UTnInd}^{\ast} }\big)_{n \ge 0}$; however this fact seems not to have any representation theoretic significance beyond its relation to the proof above.
\end{enumerate}
\end{rems}

\section{Connections to Hessenberg varieties}
\label{sec:Hessenberg}


This section will describe the relationship between the characters $\ind^{\GL_{n}}_{\UT_{\gamma}}(\mathbbm{1})$ in Theorem~\ref{thm:CQSbijection}, certain Hessenberg varieties over $\FF_{q}$, and the analogous Hessenberg varieties over $\CC$.  
These results follow a short overview of Hessenberg varieties.  
Throughout, the algebraic groups defined over $\FF_{q}$ in Section~\ref{sec:sct} and their analogues over $\CC$ are used, so the underling field will be explicitly written for each such group to avoid confusion.  

Take a field $\KK \in \{\FF_{q}, \CC\}$, and for $n \ge 0$ let $B_{n}(\KK)$ denote the subgroup of upper triangular matrices in $\GL_{n}(\KK)$.  
For each subspace $M \subseteq \Mat_{n}(\KK)$ which is stable under conjugation by elements of $B_{n}(\KK)$ and each matrix $A \in \Mat_{n}(\KK)$, the \emph{Hessenberg variety} associated to $A$ and $M$ is
\[
\BBB^{M}_{A} = \{g B_{n}(\KK) \in \GL_{n}(\KK)/B_{n}(\KK) \;|\; g^{-1} A g \in M\}.
\]
This is a slight variation---apparently due to~\cite{Tymoczko}---of the original definition in~\cite{DPS}, which requires that $M$ contain all upper triangular matrices.  The generalization is crucial, since the following results exclusively concern Hessenberg varieties associated to strictly upper triangular subspaces known as \emph{ad-nilpotent ideals}.  For $\gamma \in \III\GGG_{n}$, let
\begin{align*}
\mathfrak{ut}_{\gamma}(\KK) &= \{A \in \Mat_{n}(\KK) \;|\; \text{$A_{i, j} \neq 0$ only if $i < j$ and $(i, j) \notin \gamma$}\} \\
&= \UT_{\gamma}(\KK) - 1.
\end{align*}
These sets are in fact ideals in the algebra (and Lie algebra) of upper triangular matrices.  
Key examples of the Hessenberg varieties of the form $\BBB^{\mathfrak{ut}_{\gamma}(\KK)}_{A}$ have been known for some time, but a specific study of these varieties is quite recent; see~\cite{JiPrecup, PrecupSommers} and references therein.

\begin{prop}
\label{prop:HessenbergPoints}
Let $n \ge 0$ and $\gamma \in \III\GGG_{n}$.  For any $A \in \Mat_{n}(\FF_{q})$ with $1 + A \in \GL_{n}(\FF_{q})$,
\[
\ind^{\GL_{n}(\FF_{q})}_{\UT_{\gamma}(\FF_{q})}(\mathbbm{1})(1+A) = (q-1)^{n} q^{|E(\gamma)|} |\BBB^{\mathfrak{ut}_{\gamma}(\FF_{q})}_{A}|.
\]
\end{prop}
\begin{proof}
The proof will compute the left side of the equation directly.  Equation~\eqref{eq:permtoind} and the standard formula for induced character values give
\[
\ind^{\GL_{n}(\FF_{q})}_{\UT_{\gamma}(\FF_{q})}(\mathbbm{1})(1+A) = 
 | \{  \text{$h\UT_{\gamma}(\FF_{q}) \in \GL_{n}(\FF_{q}) / \UT_{\gamma}(\FF_{q})$} \;|\; h^{-1}(1+A)h \in \UT_{\gamma}(\FF_{q}) \}|.
\]
Each left $B_{n}(\FF_{q})$ coset in $\GL_{n}(\FF_{q})$ comprises $q^{|E(\gamma)|} (q-1)^{n}$ left $\UT_{\gamma}(\FF_{q})$ cosets, and for each $h \UT_{\gamma}(\FF_{q}) \subseteq g B_{n}(\FF_{q})$, it is the case that $h^{-1}(1+A)h \in \UT_{\gamma}(\FF_{q})$ if and only if $g^{-1}(1+A)g \in \UT_{\gamma}(\FF_{q})$:  $\UT_{\gamma}(\FF_{q})$ is normalized by $B_{n}(\FF_{q})$.  Finally, $g^{-1}(1+A)g \in \UT_{\gamma}(\FF_{q})$ if and only if $g^{-1}(1+A)g \in \mathfrak{ut}_{\gamma}(\FF_{q})$.
\end{proof}

Together with Theorem~\ref{thm:CQSbijection}, Proposition~\ref{prop:HessenbergPoints} points to a relationship between the chromatic quasisymmetric function and Hessenberg varieties for ad-nilpotent ideals over $\FF_{q}$.  More precisely, define a polynomial $d_{\lambda}^{\gamma}(t)$ by
\begin{equation}
\label{eq:dlambdadef}
X_{\gamma}(\mathbf{x}; t) = \sum_{\lambda \in \PPP_{n}} d_{\lambda}^{\gamma}(t) \widetilde{P_{\lambda}}(\mathbf{x}; t),
\end{equation}
where $\widetilde{P_{\lambda}}(\mathbf{x}; t)$ is the modified Hall--Littlewood polynomial from Section~\ref{sec:cfunsupp}.   
The aforementioned results show that $d_{\lambda}^{\gamma}(q) =  q^{|E(\gamma)|} |\BBB^{\mathfrak{ut}_{\gamma}(\FF_{q})}_{A}|$.  

The polynomials $d_{\lambda}^{\gamma}(t)$ also appear in the complex geometry of Hessenberg varieties for ad-nilpotent subspaces in a manner discovered by Precup and Sommers in~\cite{PrecupSommers}.  For the following theorem, note that the similarity classes of nilpotent matrices over any field are indexed by partitions of $n$: the class indexed by $\lambda \in \PPP_{n}$ consists of elements similar to $J_{\lambda} - 1$ as in Section~\ref{sec:cfunsupp}.

\begin{thm}[{\cite[Equation (4.7)]{PrecupSommers}}]
\label{thm:precupsommers}
For $n \ge 0$, take $\gamma \in \III\GGG_{n}$ and $\lambda \in \PPP_{n}$.  Then 
\[
\sum_{k \ge 0} \beta_{k}^{\lambda} t^{k/2} =  t^{-|E(\gamma)|} d_{\lambda}^{\gamma}(t),
\]
where $\beta_{k}^{\lambda}$ denotes the $k$th Betti number of $\BBB^{\mathfrak{ut}_{\gamma}(\CC)}_{A}$ for any nilpotent matrix $A \in \Mat_{n}(\CC)$ in the similarity class indexed by $\lambda$.
\end{thm}

Applying the map $\canoUS^{-1}$ defined in Section~\ref{sec:cfunsupp} to Equation~\eqref{eq:dlambdadef}, the left and right hand sides are computed by Theorem~\ref{thm:CQSbijection} and Proposition~\ref{prop:HessenbergPoints} respectively, giving the following.

\begin{cor}
\label{cor:adnilpotentpoincare}
For $n \ge 0$, take $\gamma \in \III\GGG_{n}$ and $\lambda \in \PPP_{n}$.  Let $A \in \Mat_{n}(\FF_{q})$ be a nilpotent elements in similarity class indexed by $\lambda$.  Then
\[
\sum_{k \ge 0} \beta_{k}^{\lambda} q^{k/2}
= |\BBB^{\mathfrak{ut}_{\gamma}(\FF_{q})}_{A}|,
\]
where the numbers $\beta_{k}^{\lambda}$ are as in Theorem~\ref{thm:precupsommers}.  
\end{cor}
\begin{proof}
Both expressions are equal to $q^{-|E(\gamma)|} (q-1)^{-n} \ind^{\GL_{n}(\FF_{q})}_{\UT_{\gamma}(\FF_{q})}(\mathbbm{1})(1 + A)$.
\end{proof}

\begin{rems}
\begin{enumerate}[label=(R\arabic*)]
\item Aside from this paper, I am aware of two works about Hessenberg varieties over $\FF_{q}$.  
The preprint~\cite{EsccobarPrecupShareshian} concerns the Hessenberg variety associated to a split regular element of $\GL_{n}(\FF_{q})$ and a subspace containing all upper triangular matrices; under some nontrivial assumptions on $q$ a result similar to Corollary~\ref{cor:adnilpotentpoincare} is established. 
This generalizes Fulman's use of Weil conjecture machinery on a subset of smooth Hessenberg varieties in order prove some identities on $q$-Eulerian numbers~\cite{Fulman}.

\item In~\cite[]{JiPrecup}, Ji and Precup give a combinatorial formula for the polynomials $d_{\lambda}^{\gamma}(t)$ by constructing an affine paving of $\BBB^{\mathfrak{ut}_{\gamma}(\CC)}_{A}$.  Precup~\cite{PrecupPC} has also suggested that a second proof of Corollary~\ref{cor:adnilpotentpoincare} could be obtained from a careful study of this paving, which would independently re-prove Theorem~\ref{thm:CQSbijection}.

\end{enumerate}
\end{rems}

\section{The vertical strip LLT polynomial as a $\GL_{n}$ character}
\label{sec:mainresult2}

This section gives a second result of the same type as Theorem~\ref{thm:CQSbijection}, in that it interprets a family of $t$-graded symmetric functions as the images of certain $\GL_{n}$ characters obtained by induction from $\UT_{n}$ under a particular isomorphism; see Table~\ref{table:theoremcomparison}.  
%
%
%
%
%
Here, the initial $\UT_{n}$ characters come from a larger set $\{\psi^{\sigma} \;|\; \sigma \in \TTT\SSS\}$ indexed by the set of tall Schr\"{o}der paths $\TTT\SSS$ from Section~\ref{sec:Catalan}, the map to $\Sym$ is a homomorphism $\canoUC: \cf(\GL_{\bullet}) \shortto \Sym$ which records the unipotent constituent of a character, and the symmetric functions are the vertical strip LLT polynomials $G_{\sigma}(\mathbf{x}; t)$, also indexed by the set $\TTT\SSS$.  Each object mentioned will be defined in this section.

\begin{thm}
\label{thm:LLTbijection}
Let $\sigma$ be a tall Schr\"{o}der path.  Then
\[
\canoUC \circ \ind^{\GL}_{\UT} (\psi^{\sigma}) = (q-1)^{|\diag(\sigma)|} \omega G_{\sigma}(\mathbf{x}; q),
\]
where $\diag(\sigma)$ is the set of diagonal steps in $\sigma$.
\end{thm}

\begin{table}
\renewcommand{\arraystretch}{1.3}
\begin{center}
\begin{tabular}{c || c || c}
& Theorem~\ref{thm:CQSbijection} & Theorem~\ref{thm:LLTbijection} \\ \hline \hline
Indexing set & \makecell{Natural unit interval \\ graphs $\gamma \in \III\GGG_{n}$} & \makecell{Tall Schr\"{o}der \\ paths $\sigma \in \TTT\SSS_{n}$} \\ \hline
$\UT_{n}$-characters & Permutation characters $\permchar^{\gamma}$  & Pseudosupercharacters $\psi^{\sigma}$ \\ \hline
Symmetric functions & \makecell{Chromatic quasisymmetric  \\ functions $X_{\gamma}(\mathbf{x}; t)$} &  \makecell{Vertical strip LLT \\ polynomials $G_{\sigma}(\mathbf{x}; t)$} \\ \hline
Map to $\Sym$ & $\canoUS: \cfunsupp(\GL_{\bullet}) \to \Sym$ & $\canoUC: \cf(\GL_{\bullet}) \to \Sym$ \\ \hline
\makecell{Meaning of \\ map to $\Sym$} & \makecell{Records unipotently supported \\ $\GL_{n}$-class functions} & \makecell{Records the irreducible \\ unipotent constituents}
\end{tabular}
\end{center}
\caption{A comparison of the results of Theorems~\ref{thm:CQSbijection} and~\ref{thm:LLTbijection} in degree $n$.}
\label{table:theoremcomparison}
\end{table}

I will now describe the meaning of this result in greater depth and outline its proof.  An irreducible character of $\GL_{n}$ is \emph{unipotent} if it is a constituent of $\ind^{\GL_{n}}_{B_{n}}(\mathbbm{1})$, where $B_{n} = B_{n}(\FF_{q})$ is the subgroup of upper triangular matrices in $\GL_{n}$.  The space
\[
\cfunchar(\GL_{\bullet}) = \CC\spanning\{\text{irreducible unipotent characters of $\GL_{n}$, $n \ge 0$}\}
\]
is a sub- and quotient Hopf algebra of $\cf(\GL_{\bullet})$, and moreover is isomorphic to $\Sym$.  Consequently, there is a homomorphism $\canoUC \colon \cf(\GL_{\bullet}) \shortto \Sym$ obtianed by projecting onto $\cfunchar(\GL_{\bullet})$ and then applying the aforementioned isomorphism, as in the diagram
\begin{equation}
\label{eq:UCdiagram}
\begin{tikzpicture}[baseline = -1.2cm]
\node at (0, 0) (GL) {$\cf(\GL_{\bullet})$};
\node at (2, -2) (GLunchar) {$\cfunchar(\GL_{\bullet})$};
\node at (4, 0) (Sym) {$\Sym$};
\draw[thick, -latex] (GL) -- node[above] {$\canoUC$} (Sym);
\begin{scope}[transform canvas={xshift = -0.5cm}]
\draw[>=latex, thick, ->>] (GL) -- (GLunchar);
\end{scope}
\begin{scope}[transform canvas={xshift=0cm}]
\draw[thick, left hook-latex] (GLunchar) -- (GL);
\end{scope}
\draw[thick, latex-latex] (GLunchar) -- node[below right] {$\cong$} (Sym); 
\end{tikzpicture}
\end{equation}
of Hopf algebra homomorphisms.  The map $\canoUC$ faithfully records the irreducible unipotent constituents of any class function of $\GL_{n}$, which can be recovered by reversing the isomorphism $\cfunchar(\GL_{\bullet}) \cong \Sym$.  Thus, Theorem~\ref{thm:LLTbijection} states that the vertical strip LLT polynomial $G_{\sigma}(\mathbf{x}; q)$ determines the irreducible unipotent constituents of the character $\ind^{\GL}_{\UT} (\psi^{\sigma})$.

An interesting connection arises from the interplay of Theorems~\ref{thm:CQSbijection} and~\ref{thm:LLTbijection}.  Carlsson and Melit~\cite[Proposition 3.5]{CM} show that for a Dyck path $\pi \in \DDD_{n}$, the plethystic relationship 
\[
(t-1)^{n} X_{\graph(\pi)}(\mathbf{x}; t)\left[\frac{\mathbf{x}}{t-1}\right] = G_{\pi}(\mathbf{x}; t)
\]
holds, where $\graph(\pi)$ is the indifference graph associated to $\pi$ in Section~\ref{sec:Catalan}.  It is also known~\cite[\nopp IV.4]{Mac} that the composite map
\[
\Sym \xrightarrow{\;\; \canoUS^{-1} \;\;} \cfunsupp(\GL_{\bullet}) \xhookrightarrow{\qquad} \cf(\GL_{\bullet}) \xrightarrow{\;\; \canoUC \;\;} \Sym
\]
is an isomoprhism which can be expressed in plethystic notation as $f[\mathbf{x}] \shortmapsto \omega f[\frac{\mathbf{x}}{t-1}]|_{t = q}$, so my results give a $\GL_{n}$-representation theoretic interpretation of Carlsson and Melit's result; at the same time,~\cite[Proposition 3.5]{CM} could be used to prove Theorem~\ref{thm:LLTbijection} via Theorem~\ref{thm:CQSbijection}.

The proof of Theorem~\ref{thm:CQSbijection} in this paper will instead use the machinery of combinatorial Hopf algebras, which has the benefit of giving a new description of the map $\canoUC \circ \ind^{\GL}_{\UT}$.  Define a Hopf algebra homomorphism $\canoLLT\colon \scf(\UT_{\bullet}) \shortto \QSym$ as the composite map in the diagram
\begin{equation}
\label{eq:LLTdiagram}
\begin{tikzpicture}[baseline = -1.2cm]
\node at (0, 0) (UT) {$ \scf(\UT_{\bullet})$};
\node at (0, -2) (GLunsupp) {$\cfunsupp(\GL_{\bullet})$};
\node at (4, -2) (GL) {$\cf(\GL_{\bullet})$};
\node at (8, -2) (Sym) {$\Sym$};
\node at (12, -2) (QSym) {$\QSym$};
\draw[thick, -latex] (UT) -- node[left] {$\ind^{\GL}_{\UT}$} (GLunsupp); 
\draw[thick, right hook-latex] (GLunsupp) -- (GL);
\draw[thick, -latex] (GL) -- node[below] {$\canoUC$} (Sym);
\draw[thick, right hook-latex] (Sym) -- (QSym);
\draw[thick, dashed, -latex] (UT) -- node[above] {$\canoLLT$} (QSym); 
\end{tikzpicture}
\end{equation}
of Hopf algebras, so that Theorem~\ref{thm:LLTbijection} describes $\canoLLT$ implicitly.  By definition, $\canoLLT$ can be computed by inducing a character of $\UT_{n}$ to $\GL_{n}$ and recording its unipotent constituents as symmetric functions.  However, Theorem~\ref{thm:universalhopfmap} shows that $\canoLLT$ is also determined by the linear character $\operatorname{ps}_{1} \circ \canoLLT$ of the Hopf algebra $\scf(\UT_{\bullet})$.  It happens that this linear character coincides exactly with one defined by Guay-Paquet, so that a result of~\cite{GP16}---restated in Corollary~\ref{cor:GPidentityLLT}---shows that
\begin{equation}
\label{eq:GPLLTidentity}
\canoLLT(\permind_{\graph(\pi)}) = G_{\pi}(\mathbf{x}; q^{-1}) \qquad\text{for $\pi \in \DDD$}.
\end{equation}
Finally, several known identities for LLT polynomials complete the proof; these are given in Proposition~\ref{prop:LLTsieve}.

The remainder of the section is divided into three parts.  First, Section~\ref{sec:psisigma} describes the characters $\psi^{\sigma}$ appearing in Theorem~\ref{thm:LLTbijection} and shows that this family includes both the permutation characters and supercharacters of $\scf(\UT_{\bullet})$.  Then, Section~\ref{sec:cfunchar} describes the map $\canoLLT$ as a CHA morphism to $(\QSym, \operatorname{ps}_{1})$, defining the necessary combinatorial Hopf algebra structures on $\scf(\UT_{\bullet})$ and $\cf(\GL_{\bullet})$ along the way.  Finally, Section~\ref{sec:LLT} formally defines the vertical strip LLT polynomial, shows how it can be realized as the image of a CHA morphism, and concludes with a proof of Theorem~\ref{thm:LLTbijection}.

\begin{rem}
\label{rem:LLTnofactor}
It is possible to ``remove'' the factors of $q-1$ in Theorem~\ref{thm:LLTbijection}.  
With results in Section~\ref{sec:psisigma}, work of Andrews and Thiem~\cite[Remark on p.~490]{AnTh} and Aliniaeifard and Thiem~\cite[Remark (1) on p.~13]{AlTh20} show that each $\psi^{\sigma}$ is the sum of $(q-1)^{|\diag(\sigma)|}$ distinct characters which each have the same image under $\canoUC \circ \ind^{\GL}_{\UT}$; this image must be $\omega G_{\sigma}(\mathbf{x}; q)$.
\end{rem}

\subsection{The pseudosupercharacters $\psi^{\sigma}$}
\label{sec:psisigma}

This section will define the characters $\psi^{\sigma}$ appearing in Theorem~\ref{thm:LLTbijection}.  Recall the terminology used for Schr\"{o}der paths in Section~\ref{sec:Catalan} and the characters of $\UT_{n}$ defined in Section~\ref{sec:sct}.

For $\sigma \in \TTT\SSS_{n}$, the \emph{pseudosupercharacter} indexed by $\sigma$ is the class function
\[
\psi^{\sigma} = \sum_{S \subseteq \diag(\sigma)} (-1)^{|\diag(\sigma) \setminus S|} \permchar^{\big([n],\, \area(\sigma) \cup S \big)} \in \scf(\UT_{\bullet}).
\]
The definition of $\diag(\sigma)$ ensures that each graph $([n], \area(\sigma) \cup S)$ above is in fact an indifference graph.  For example, with
\begin{equation}
\label{eq:pseudosupercharacterexample}
\sigma = \begin{tikzpicture}[scale = 0.45, baseline = 0.45*-1.7cm]
\draw[gray] (-0.2, 0.2) grid (3.2, -3.2);
\draw[dashed, gray] (-0.2, 0.2) -- (3.2, -3.2);
\draw[very thick] (0, 0) -- (1, 0) -- (2, -1)  -- (3, -1) -- (3, -3);
\end{tikzpicture}
\qquad\text{we have}\qquad
\psi^{\sigma} = - \permchar^{\begin{tikzpicture}[scale = 0.35, baseline = 0.35*-0.35cm]
\draw[fill] (0, 0) circle (2pt) node[inner sep = 1pt] (1) {};
\draw[fill] (1, 0) circle (2pt) node[inner sep = 1pt] (2) {};
\draw[fill] (2, 0) circle (2pt) node[inner sep = 1pt] (3) {};
\node[below] at (1) {$\scriptscriptstyle 1$};
\node[below] at (2) {$\scriptscriptstyle 2$};
\node[below] at (3) {$\scriptscriptstyle 3$};\draw (2) --  (3);
\end{tikzpicture}\hspace{-0.25em}} + \permchar^{\hspace{-0.25em}\begin{tikzpicture}[scale = 0.35, baseline = 0.35*-0.35cm]
\draw[fill] (0, 0) circle (2pt) node[inner sep = 1pt] (1) {};
\draw[fill] (1, 0) circle (2pt) node[inner sep = 1pt] (2) {};
\draw[fill] (2, 0) circle (2pt) node[inner sep = 1pt] (3) {};
\node[below] at (1) {$\scriptscriptstyle 1$};
\node[below] at (2) {$\scriptscriptstyle 2$};
\node[below] at (3) {$\scriptscriptstyle 3$};
\draw (1) -- (2);
\draw (2) --  (3);
\end{tikzpicture}\hspace{-0.25em}}.
\end{equation}
A noteworthy family of examples is the pseudosupercharacters indexed by Dyck paths: for $\pi \in \DDD$, $\diag(\pi) = \emptyset$, from which it follows that
\[
\psi^{\pi} = \permchar^{\graph(\pi)}.
\]


\begin{prop}
\label{prop:psidecomp}
Let $\sigma$ be a tall Schr\"{o}der path of size $n \ge 0$.  Then $\psi^{\sigma}$ is a character, and in particular
\[
\psi^{\sigma} = \hspace{-1em} \sum_{\substack{ E(\gamma) \subseteq (\area(\sigma) \cup \diag(\sigma)) \\[0.2em] \diag(\sigma) \subseteq E(\gamma) }} \hspace{-1em} \chi^{\rho},
\]
where the sum is over indifference graphs $\gamma \in \III\GGG_{n}$ satisfying the given conditions.  
\end{prop}
\begin{proof}
Using the definition of $\psi^{\sigma}$, 
\[
\psi^{\sigma} =  \sum_{S \subseteq \diag(\sigma)} (-1)^{|\diag(\sigma) \setminus S|} \hspace{-1em} \sum_{E(\gamma) \subseteq \area(\sigma) \cup  S }  \hspace{-1em} \chi^{\gamma},
\]
where the sum is over indifference graphs $\gamma$ as in the proposition.  Combining like terms in the sum above, we obtain
\[
\psi^{\sigma} = \sum_{E(\gamma) \subseteq \area(\sigma) \cup \diag(\sigma) } \left( \sum_{T \supseteq E(\gamma) \cap \diag(\sigma)}  \hspace{-1em} (-1)^{|\diag(\sigma) \setminus T|} \right) \chi^{\gamma}.
\]
The proposition now follows from the binomial theorem.
\end{proof}

As an example of Proposition~\ref{prop:psidecomp}, the pseudosupercharacter in Equation~\ref{eq:pseudosupercharacterexample} expands as the sum of supercharacters
\[
\psi^{\sigma} = 
\chi^{\begin{tikzpicture}[scale = 0.35, baseline = 0.35*-0.35cm]
\draw[fill] (0, 0) circle (2pt) node[inner sep = 1pt] (1) {};
\draw[fill] (1, 0) circle (2pt) node[inner sep = 1pt] (2) {};
\draw[fill] (2, 0) circle (2pt) node[inner sep = 1pt] (3) {};
\node[below] at (1) {$\scriptscriptstyle 1$};
\node[below] at (2) {$\scriptscriptstyle 2$};
\node[below] at (3) {$\scriptscriptstyle 3$};
\draw (1) -- (2);
\end{tikzpicture}\hspace{-0.25em}} 
+ \chi^{\begin{tikzpicture}[scale = 0.35, baseline = 0.35*-0.35cm]
\draw[fill] (0, 0) circle (2pt) node[inner sep = 1pt] (1) {};
\draw[fill] (1, 0) circle (2pt) node[inner sep = 1pt] (2) {};
\draw[fill] (2, 0) circle (2pt) node[inner sep = 1pt] (3) {};
\node[below] at (1) {$\scriptscriptstyle 1$};
\node[below] at (2) {$\scriptscriptstyle 2$};
\node[below] at (3) {$\scriptscriptstyle 3$};
\draw (1) -- (2);
\draw (2) --  (3);
\end{tikzpicture}\hspace{-0.25em}}.
\]

The final result in the section shows that every supercharacter of $\scf(\UT_{n})$ occurs as a pseudosupercharacter.  Given a Dyck path $\pi$, a \emph{peak} of $\pi$ is a sequence of steps $ES$; say that a peak is \emph{tall} if the first step $E$ does not begin on the diagonal $x= -y$.  For example,
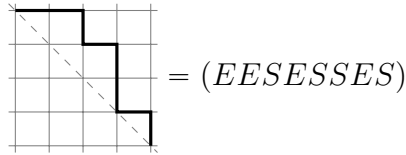
\[
\begin{tikzpicture}[scale = 0.45, baseline = 0.45*-2.2cm]
\draw[gray] (-0.2, 0.2) grid (4.2, -4.2);
\draw[dashed, gray] (-0.2, 0.2) -- (4.2, -4.2);
\draw[very thick] (0, 0) -- (2, 0) -- (2, -1) -- (3, -1) -- (3, -3) -- (4, -3) -- (4, -4);
\end{tikzpicture} = (EESESSES)
\]
has three peaks, but only two tall peaks.  Define the \emph{Mesa path} of $\pi\in \DDD_{n}$ to be the tall Schr\"{o}der path $\mesa(\pi) \in \TTT\SSS_{n}$ obtained by first constructing $\Dyck(\pi)$ and then replacing each tall peak $ES$ with a diagonal step $D$; for example
\[
\mesa\left( \begin{tikzpicture}[scale = 0.45, baseline = 0.45*-2.2cm]
\draw[gray] (-0.2, 0.2) grid (4.2, -4.2);
\draw[dashed, gray] (-0.2, 0.2) -- (4.2, -4.2);
\draw[very thick] (0, 0) -- (2, 0) -- (2, -1) -- (3, -1) -- (3, -3) -- (4, -3) -- (4, -4);
\end{tikzpicture} \right) = 
\begin{tikzpicture}[scale = 0.45, baseline = 0.45*-2.2cm]
\draw[gray] (-0.2, 0.2) grid (4.2, -4.2);
\draw[dashed, gray] (-0.2, 0.2) -- (4.2, -4.2);
\draw[very thick] (0, 0) -- (1, 0) -- (3, -2) -- (3, -3) -- (4, -3) -- (4, -4);
\end{tikzpicture} = (EDDSES).
\]

\begin{prop}
\label{prop:scLLT}
Let $\pi$ be a Dyck path.  Then $\psi^{\mesa(\pi)} = \chi^{\graph(\pi)}$.
\end{prop}
\begin{proof}
By assumption,
\[
\area(\pi) = \area(\mesa(\pi)) \cup \diag(\mesa(\pi)),
\]
so by Proposition~\ref{prop:psidecomp}, 
\[
\psi^{\mesa(\pi)} = \sum_{\substack{ \gamma \subseteq \graph(\pi) \\[0.2em] \diag(\mesa(\pi)) \subseteq E(\gamma) }} \chi^{\gamma}.
\]
Now suppose that an indifference graph $\gamma$ is a proper spanning subgraph of $\graph(\pi)$.  Then $\gamma$ must be missing at least one edge $\{i, j\}$ such that the unit square indexed by $\{i, j\}$ is bordered directly by a tall peak of $\pi$, so that $\{i, j\} \in  \diag(\mesa(\pi))$, and $\chi^{\gamma}$ does not appear in the sum above.  Thus the only summand above is $\chi^{\graph(\pi)}$.
\end{proof}

\subsection{Factoring $\canoLLT$ through $\cf(\GL_{\bullet})$}
\label{sec:cfunchar}

Recall the discussion of unipotent characters of $\GL_{n}$ at the outset of Section~\ref{sec:mainresult2}.  The irreducible unipotent characters of $\GL_{n}$ are indexed by the partitions of $n$, with the character corresponding to $\lambda \in \PPP(n)$ written $\chi^{\lambda}$.  I will follow the convention of~\cite{Mac} in which $\chi^{(1^{n})}$ is the trivial character $\mathbbm{1}$ of $\GL_{n}$ and $\chi^{(n)}$ is the \emph{Steinberg character} $\operatorname{St}_{n}$; this differs from the convention of~\cite{Zel} and others by the transposition of each partition.

The homomorphism $\canoUC$ was constructed by Zelevinksy~\cite[\nopp 9.4]{Zel}, and is given by
\begin{equation}
\label{eq:canoUCdef}
\begin{array}{rccc}
\canoUC\colon & \cf(\GL_{\bullet}) & \to & \Sym \\[0.25em]
& \psi & \mapsto & \sum_{\lambda} \langle \psi,  \chi^{\lambda} \rangle s_{\lambda}.
\end{array}
\end{equation}
This map has right inverse given by $s_{\lambda} \shortmapsto \chi^{\lambda}$, which is also a Hopf algebra homomorphism.  Thus, $\cf(\GL_{\bullet})$ has a quotient and sub-Hopf algebra
\[
\cfunchar = \CC\spanning\{\chi^{\lambda} \;|\; \lambda \in \PPP\},
\]
through which $\canoUC$ factors, as shown in Diagram~\eqref{eq:UCdiagram}.

By Theorem~\ref{thm:universalhopfmap}, the map $\canoUC$ is equivalent to a linear character of the Hopf algebra $\cf(\GL_{\bullet})$.  This character is also given in~\cite{Zel}, and is
\[
\ket{\operatorname{St}_{\bullet}} = \big( \ket{\operatorname{St}_{n}} \big)_{n \ge 0} \in \cf(\GL_{\bullet})^{\ast}.
\]

\begin{prop}[{\cite[\nopp 9.4--5]{Zel}}]
The map $\ket{\operatorname{St}_{\bullet}}$ is a linear character of $\cf(\GL_{\bullet})$, and moreover $\canoUC$ is the unique CHA morphism $(\cf(\GL_{\bullet}), \ket{\operatorname{St}_{\bullet}}) \shortto (\Sym, \operatorname{ps}_{1})$.
\end{prop}

Now, for $n \ge 0$, write $\reg_{\UT_{n}}$ for the regular character of $\UT_{n}$.  Define a linear functional 
\[
\ket{\reg_{\bullet}} = \big( \ket{\reg_{\UT_{n}}} \big)_{n \ge 0} \in \scf(\UT_{\bullet})^{\ast}.
\]

\begin{thm}
\label{thm:canoLLT}
The function $\ket{\reg_{\bullet}}$ is linear character of $\scf(\UT_{\bullet})$; moreover, 
\[
\ket{\reg_{\bullet}} = \ket{\operatorname{St}_{\bullet}} \circ \ind^{\GL}_{\UT},
\]
so $\ind^{\GL}_{\UT}$ is a CHA morphism 
\[
\big( \scf(\UT_{\bullet}), \ket{\reg_{\bullet}} \big) \xrightarrow{\;\; \ind^{\GL}_{\UT} \;\;} \big( \cf(\GL_{\bullet}), \ket{\operatorname{St}_{\bullet}} \big).
\]
\end{thm}
\begin{proof}
It is sufficient to prove that $\ket{\reg_{\bullet}} = \ket{\operatorname{St}_{\bullet}} \circ \ind^{\GL}_{\UT}$.  Doing so requires the well-known fact (see e.g.~\cite[\nopp 10.3]{Zel}) that for unipotent $X \in \GL_{n}$, 
\[
\operatorname{St}_{n}(X) = \begin{cases} |\UT_{n}| & \text{if $X = 1_{n}$} \\ 0 & \text{for other unipotent $X$.} \end{cases}
\]
As a consequence,
\[
\res^{\GL_{n}}_{\UT_{n}}(\operatorname{St}_{n}) = \reg_{\UT_{n}}.
\]
With this, the claim follows from Frobenius reciprocity as described in Section~\ref{sec:repHopf}:
\[
\ket{\operatorname{St}_{\bullet}} \circ \ind^{\GL}_{\UT} = \big( \ket{ \res^{\GL_{n}}_{\UT_{n}}(\operatorname{St}_{n})} \big)_{n \ge 0} = \ket{\reg_{\bullet}}. \qedhere
\]
\end{proof}

\begin{rem}
Like Theorem~\ref{thm:canoCQS}, Theorem~\ref{thm:canoLLT} actually shows that $\ind^{\GL}_{\UT}$ is a CHA morphism from the larger combiantorial Hopf algebra $( \cf(\UT_{\bullet}), \ket{\reg_{\bullet}} )$ to $(\QSym, \operatorname{ps}_{1})$.
\end{rem}

\subsection{The vertical strip LLT polynomial}
\label{sec:LLT}

The \emph{vertical strip LLT polynomial} indexed by a tall Schr\"{o}der path $\sigma$ is
\[
G_{\sigma}(\mathbf{x}; t)  = 
 \sum_{\kappa \in A(\sigma)} t^{\asc_{\left([n], \area(\sigma)\right)}(\kappa)} x_{\kappa(1)} x_{\kappa(2)} \dots x_{\kappa(n)} \in \CC[[\mathbf{x}]][t],
\]
where the sum is over the set $A(\sigma)$ of functions $\kappa: [n] \shortto \ZZ_{> 0}$ which satisfy $\kappa(i) < \kappa(j)$ for each $i < j$ with $\{i, j\} \in \diag(\sigma)$.  
Viewed as a polynomial in $t$, the coefficients of $G_{\sigma}(\mathbf{x}; t)$ are actually symmetric functions~\cite[Lemma 10.2]{HHL}, though this is not obvious.  For example, 
\[
G_{\begin{tikzpicture}[scale = 0.15, baseline = 0.15*-1.7cm]
\draw[gray] (-0.2, 0.2) grid (3.2, -3.2);
\draw[dashed, gray] (-0.2, 0.2) -- (3.2, -3.2);
\draw[very thick] (0, 0) -- (2, 0)  -- (3, -1) -- (3, -3);
\end{tikzpicture}}(\mathbf{x}; t) =  t\;m_{(2, 1)} + (t^{2} + 2t)\;m_{(1^{3})}.
\]

\begin{rem}
There are several essentially equivalent definitions of LLT polynomials; the one above is due to~\cite{CM} in the unicellular case and to~\cite{Dadd} (see also~\cite{AP}) in general.
\end{rem}

If $\sigma$ is a Dyck path, so that $\diag(\sigma) = \emptyset$, then the sum in $G_{\sigma}(\mathbf{x}; t)$ is over all possible colorings; this special case is know as a \emph{unicellular} LLT polynomial.  In~\cite{GP16}, Guay-Paquet realizes the unicellular LLT polynomials by way of a homomorphism of Hopf algebras over $\CC[t]$.  By evaluating at $t = q^{-1}$ as in Theorem~\ref{thm:scfNG}, this result descends to a Hopf algebra homomorphism $\scf(\UT_{\bullet}) \shortto \QSym$.  Define a linear functional
\[
\begin{array}{rcl}
\zeta_{1}: \scf(\UT_{\bullet}) & \to & \CC \\[0.5em]
\permind_{\gamma} & \mapsto & 1.
\end{array}
\]

\begin{thm}[{\cite[Theorem~57]{GP16}}]
\label{thm:GPcano2}
The map $\zeta_{1}$ is a linear character of $\scf(\UT_{\bullet})$, and the unique CHA morphism
\[
\big( \scf(\UT_{\bullet}), \zeta_{1} \big) \to (\QSym, \operatorname{ps}_{1})
\]
is given by
\[
\permind_{\graph(\pi)} \mapsto G_{\pi}(\mathbf{x}; q^{-1})
\qquad\text{for $\pi \in \DDD$}.
\]
\end{thm}

Now recall the character $ \ket{\reg_{\bullet}}$ defined in the previous section.

\begin{prop}
As a linear character of the Hopf algebra $\scf(\UT_{\bullet})$, $\ket{\reg_{\bullet}}$ is equal to $\zeta_{1}$; in particular
\[
\ket{\reg_{\bullet}}(\permind_{\gamma}) = 1 \qquad \text{for $\gamma \in \III\GGG$}.
\]
\end{prop}
\begin{proof}
This follows from direct computation: if $\gamma \in \III\GGG_{n}$, 
\[
\ket{\reg_{\bullet}}(\permind_{\gamma}) = \langle \permind_{\gamma}, \reg_{\UT_{n}} \rangle = \permind_{\gamma}(1_{n}) = 1. \qedhere
\]
\end{proof}

The uniqueness result of Theorem~\ref{thm:universalhopfmap} now gives the following, which restates Equation~\ref{eq:GPLLTidentity}.

\begin{cor}
\label{cor:GPidentityLLT}
The map $\canoLLT$ is the CHA morphism described in Theorem~\ref{thm:GPcano2}.  In particular,
\[
\canoLLT(\permind_{\graph(\pi)}) = G_{\pi}(\mathbf{x}; q^{-1}) \qquad\text{for $\pi \in \DDD$}.
\]
\end{cor}

\begin{rem}[c.f.~Remark~\ref{rem:CQSsymmetry}~(R1)]
Corollary~\ref{cor:GPidentityLLT} can be used to give a novel proof that the unicellular LLT polynomial $G_{\pi}(\mathbf{x}; t)$ has symmetric coefficients.
\end{rem}

The proof of Theorem~\ref{thm:LLTbijection} is given below following two identities for LLT polynomials.

\begin{prop}[{\cite[Theorem~2.1]{AS}} and {\cite[Proposition 3.4]{CM}}]
\label{prop:LLTsieve}
\label{prop:LLTtransform}
Let $n$ be a positive integer.
\begin{enumerate}[label = (\roman*)]
\item For any Dyck paths $\pi \in \DDD_{n}$, 
\[
q^{|\area(\pi)|} G_{\Dyck(\pi)}(\mathbf{x}; q^{-1}) = \omega G_{\Dyck(\pi)}(\mathbf{x}; q).
\]

\item For any tall Schr\"{o}der paths $\sigma \in \TTT\SSS_{n}$, 
\[
(q-1)^{|\diag(\sigma)|} G_{\sigma}(\mathbf{x}; q) = \sum_{S \subseteq \diag(\sigma)} (-1)^{|\diag(\sigma) \setminus S|} G_{
\area^{-1}\big(  \area(\sigma) \cup S \big)}(\mathbf{x}; q),
\]
where $\area^{-1}\big(  \area(\sigma) \cup S \big)$ denotes the unique Dyck path with area $\area(\sigma) \cup S$.

\end{enumerate}
\end{prop}

\begin{proof}[Proof of Theorem~\ref{thm:LLTbijection}]
For $\pi \in \DDD$, Equation~\eqref{eq:permtoind} states that $\permchar^{\graph(\pi)} = q^{|\area(\pi)|} \permind_{\graph(\pi)}$, so by Proposition~\ref{prop:LLTtransform} (i),
\[
\canoLLT(\permchar^{\graph(\pi)}) = \omega G_{\pi}(\mathbf{x}; q).
\]
Combining this with Proposition~\ref{prop:LLTtransform} (ii) and the linearity of $\omega$,
\[
\canoLLT(\psi^{\sigma}) =  \sum_{S \subseteq \diag(\sigma)} (-1)^{|\diag(\sigma) \setminus S|} \omega G_{\pi + S}(\mathbf{x}; q) = (q-1)^{|\diag(\sigma)|} \omega G_{\sigma}(\mathbf{x}; q).\qedhere
\]
\end{proof}

\section{Positivity conjectures}
\label{sec:positivity}

Recall the bases of $\Sym$ given in Section~\ref{sec:QSym}.  An element $f(\mathbf{x}; t) \in \Sym[t]$ is said to be \emph{$e$-positive} if the coefficients $a_{\lambda}(t)$ in
\[
f(\mathbf{x}; t) = \sum_{\lambda \in \PPP} a_{\lambda}(t) e_{\lambda}
\]
are polynomials in $t$ with nonnegative coefficients: $a_{\lambda}(t) \in \ZZ_{\ge 0}[t]$.  Likewise, if the coefficients of $f(\mathbf{x}; t)$ in any other basis of $\Sym$ have this property---for example, the Schur basis $\{s_{\lambda} \;|\; \lambda \in \PPP\}$---say that $f(\mathbf{x}; t)$ is positive in that basis.   The positivity of the symmetric functions in this paper are of some interest.

For chromatic quasisymmetric functions in Section~\ref{sec:CQS}, $e$-positivity generalizes the Stanley--Stembridge conjecture~\cite[Conjecture 5.5]{StaStem}, which by~\cite{GP13} is the $t = 1$ case below.

\begin{conj}[{\cite[Conjecture 1.3]{ShWa}}]
\label{conj:epos}
For each $\gamma \in \III\GGG$, $X_{\gamma}(\mathbf{x}; t)$ is $e$-positive.
\end{conj}

Special cases of Conjecture~\ref{conj:epos} have have explicit solutions, as in~\cite{AbreuNigro, CMP, HaradaPrecup, HuhNamYoo}.

For vertical strip LLT polynomials in Section~\ref{sec:LLT}, Schur positivity has implications for the study of Macdonald polynomials~\cite{HHL}.  The next two statements are adapted for vertical strip LLT polynomials from the setting of general LLT polynomials.

\begin{thm}[{\cite[Corollary 6.9]{GrojHai}}]
\label{thm:LLTschurpos}
For each $\sigma \in \TTT\SSS$, $G_{\sigma}(\mathbf{x}; t)$ is positive in the Schur basis.
\end{thm}

The proof in~\cite{GrojHai} is algebraic and does not construct the Schur coefficients, although some special cases are known, including the $q$-Kostka numbers~\cite{LascSchutz} and the results of~\cite{HuhNamYoo, Tom}.

\begin{question}[{\cite[Open Problem 6.6]{HaglundBook}}]
\label{question:LLT}
Find a (manifestly positive) combinatorial formula for the Schur coefficients of $G_{\sigma}(\mathbf{x}; t)$.
\end{question}

The $e$-positivity of vertical strip LLT polynomials is also the subject of study; in this context, the paradigm is altered by substituting $t+1$ for $t$.

\begin{thm}[{\cite[Theorem~5.5]{Dadd}}]
\label{thm:LLTeposExistence}
For each $\sigma \in \TTT\SSS$, $G_{\sigma}(\mathbf{x}; t+1)$ is $e$-positive.
\end{thm}

The paper~\cite{AS} proves an explicit combinatorial formula for the $e$-coefficients in Theorem~\ref{thm:LLTeposExistence}, which will be restated in Section~\ref{sec:GLeposLLT}.  Using Theorem~\ref{thm:LLTbijection}, this formula implies a result about the characters $\ind^{\GL}_{\UT}(\psi^{\sigma})$, which inadvertently gives some representation theoretic intuition for the $t \leftrightarrow t+1$ shift above.





Returning to the general discussion of positivity, if a polynomial $f(\mathbf{x}; t) \in \Sym[t]$ is positive with respect to a certain basis, then evaluating $t$ at any positive integer will give a symmetric function with nonnegative integer coefficients in the chosen basis.  
Thus, evaluating $t = q$ above gives positivity results about the $\GL_{n}$ characters in this paper.  
Conversely, polynomial equations can be verified on any infinite set---like the set of prime powers---so $\GL_{n}$ characters offer a novel approach to some of the open problems above.  

This section reinterprets each of the positivity statements above in the context of $\GL_{n}$ representation theory.  Section~\ref{sec:GLepos} will discuss the $e$-positivity of the chromatic quasisymmetric function, Section~\ref{sec:GLspos} will discuss Schur positivity of the vertical strip LLT polynomials, and Section~\ref{sec:GLeposLLT} will discuss the implications of the $e$-positivity of vertical strip LLT polynomials.


\subsection{Interpreting the $e$-positivity of $X_{\gamma}(\mathbf{x}; t)$}
\label{sec:GLepos}

In light of Theorem~\ref{thm:CQSbijection}, there should be a restatement of Conjecture~\ref{conj:epos} involving the characters $\ind^{\GL_{n}}_{\UT_{\gamma}}(\mathbbm{1})$.  
However, the isomorphism $\canoUS$ in Theorem~\ref{thm:CQSbijection} does not associate $e_{\lambda}$ to a character of $\GL_{n}$, so some interpretation is required.  
My choice to use the particular restatement below is informed by ongoing work on the subject.

Recall the Steinberg character $\operatorname{St}_{n} \in \cf(\GL_{n})$ defined in Section~\ref{sec:cfunchar}.   For any partition $\lambda = (\lambda_{1}, \ldots, \lambda_{\ell})$, define $\operatorname{St}_{\lambda} \in \cf(\GL_{\bullet})$ to be the product
\[
\operatorname{St}_{\lambda} = \operatorname{St}_{\lambda_{1}} \operatorname{St}_{\lambda_{2}} \cdots \operatorname{St}_{\lambda_{\ell}}.
\]

\begin{conj}
\label{conj:GLepos}
Let $n \ge 0$ and $\gamma \in \III\GGG_{n}$.  There are polynomials $a_{\lambda}^{\gamma}(t) \in \ZZ_{\ge 0}[t]$
 such that for each prime power $q$ the character
\[
\eta_{\gamma} = \sum_{\lambda \in \PPP_{n}} a_{\lambda}^{\gamma}(q) \operatorname{St}_{\lambda}
\]
satisfies $(q-1)^{n} \eta_{\gamma}(u) = \ind^{\GL_{n}}_{\UT_{\gamma}}(\mathbbm{1})(u)$ for every unipotent element $u \in \GL_{n}(\FF_{q})$.
\end{conj}

\begin{prop}
Conjectures~\ref{conj:epos} and~\ref{conj:GLepos} are equivalent.
\end{prop}
\begin{proof}
For a class function $\psi \in \cf(\GL_{n})$, write $\psi|_{\text{uni}} \in \cfunsupp(\GL_{\bullet})$ for the element defined by
\[
\psi|_{\text{uni}}(g) = \begin{cases} \psi(g) & \text{if $g$ is unipotent} \\ 0 & \text{otherwise,} \end{cases}
\]
so that Conjecture~\ref{conj:GLepos} states $\sum_{\lambda \in \PPP_{n}} a_{\lambda}^{\gamma}(q) \operatorname{St}_{\lambda}|_{uni} = \frac{1}{(q-1)^{n}}\ind^{\GL_{n}}_{\UT_{\gamma}}(\mathbbm{1})$.

I now claim that $\canoUS(\operatorname{St}_{\lambda}|_{uni}) = e_{\lambda}$, so that with the preceeding remarks and Theorem~\ref{thm:CQSbijection} the proof will be complete.  
The claim is relatively well-known, but a proof sketch is included for the sake of completeness.  
Direct computation gives that $\operatorname{St}_{n}|_{\text{uni}} = q^{\binom{n}{2}} \delta_{(1^{n})}$ (see the proof of Theorem~\ref{thm:canoLLT}), and 
\[
 \canoUS\big( q^{\binom{n}{2}}  \delta_{(1^{n})} \big)  = \tilde{P}_{(1^{n})}(\mathbf{x}; q)  = e_{n},
\]
with the second equality due to the definition of the Hall--Littlewood polynomial; see~\cite[\nopp III.2 (2.8)]{Mac}.  The claim then follows from the fact~\cite[\nopp 10.1]{Zel} that the extension of $\psi \shortmapsto \psi_{\text{uni}}$ to all of  $\cf(\GL_{\bullet})$ is a Hopf algebra homomorphism to $\cfunsupp(\GL_{\bullet})$.
\end{proof}

\begin{rems}
\begin{enumerate}[label=(R\arabic*)]
\item A direct proof of Conjecture~\ref{conj:GLepos} would probably find an organic realization of the character $\eta_{\gamma}$ using the representation theory of $\GL_{n}$, and in a manner which does not depend on $q$.  
Ongoing work has identified a promising candidate for the character $\eta_{\gamma}$, but has not led to any progress on the conjecture itself.

\item It is not clear that Conjecture~\ref{conj:GLepos} offers an easier approach to Conjecture~\ref{conj:epos} than other equivalent statements.  However, as the clearest restatement of Conjecture~\ref{conj:epos} in the $\GL_{n}(\FF_{q})$ context, the wide interest in $e$-positivity seems to justify its inclusion.

\end{enumerate}
\end{rems}

\subsection{Interpreting the Schur positivity of $G_{\sigma}(\mathbf{x}; t)$}
\label{sec:GLspos}

Let $\sigma$ be a tall Schr\"{o}der path and write
\[
G_{\sigma}(\mathbf{x}; t) = \sum_{\lambda \in \PPP} b_{\lambda}^{\sigma}(t) s_{\lambda}.
\]
It is immediate that each $b_{\lambda}^{\sigma}(t)$ is a polynomial in $t$ with integral coefficients, and the content of Theorem~\ref{thm:LLTschurpos} is that the coefficients of this polynomial are nonnegative.  

Recall from Section~\ref{sec:cfunchar} that the irreducible unipotent characters of $\GL_{n}$ are $\{\chi^{\lambda} \;|\; \lambda \in \PPP_{n}\}$, and that $\canoUC(\chi^{\lambda}) = s_{\lambda}$ for each partition $\lambda \in \PPP$.  Thus, for a tall Schr\"{o}der path $\sigma$, Theorem~\ref{thm:LLTbijection} implies that 
\begin{equation}
\label{eq:GLLLTschurpos}
(q-1)^{|\diag(\sigma)|} b_{\lambda}^{\sigma}(q) = \langle \chi^{\lambda'}, \ind^{\GL}_{\UT}(\psi^{\sigma}) \rangle ,
\end{equation}
which is the multiplicity of the irreducible unipotent $\GL_{n}$-module indexed by $\lambda'$ in the $\GL_{n}$-module affording $\ind^{\GL}_{\UT}(\psi^{\sigma})$.  Thus, Theorem~\ref{thm:LLTbijection} implies the known fact that $b_{\lambda}^{\sigma}(q)$ is nonnegative for each prime power $q$, but falls short of giving a second proof of Theorem~\ref{thm:LLTschurpos}: a polynomial with negative coefficients can still take on infinitely many positive values.  Nonetheless, progress on Open Problem~\ref{question:LLT} might be obtained through explicit representation theoretic formulas.  


\begin{question}
For $n \ge 0$, $\sigma \in \TTT\SSS_{n}$, and $\lambda \in \PPP_{n}$, find a combinatorial formula for $ \langle \chi^{\lambda'}, \ind^{\GL}_{\UT}(\psi^{\sigma}) \rangle$ as a function of $q$.
\end{question}

Such a formula would almost certainly be divisible by $(q-1)^{|\diag(\sigma)|}$ in a straightforward manner; see Remark~\ref{rem:LLTnofactor}.  This would give an answer to Open Problem~\ref{question:LLT}.

\subsection{Interpreting the $e$-positivity of $G_{\sigma}(\mathbf{x}; t)$}
\label{sec:GLeposLLT}
\newcommand{\bbGamma}{\Gamma}

The final section of this paper will show how the explicit $e$-positivity formula for vertical strip LLT polynomials given in~\cite{AS} leads to a deeper understanding of the characters $\ind^{\GL}_{\UT}(\psi^{\sigma})$ from Theorem~\ref{thm:LLTbijection}; see Corollary~\ref{cor:GLeposLLT}.  I will begin by recalling the main result of~\cite{AS}.

Fix a graph $\gamma = ([n], E(\gamma))$ on $[n]$.  An \emph{orientation} of $\gamma$ is a collection of directed edges
\[
\theta = \{(i, j) \;|\; \{i, j\} \in E(\gamma)\},
\]
so that $([n], \theta)$ is a directed graph whose underlying undirected graph is $\gamma$.  For example, with
\begin{equation}
\label{eq:hrvexample}
\gamma = \begin{tikzpicture}[scale = 0.5, baseline = 0.5*-0.3cm]
\draw[fill] (0, 0) circle (2pt) node[inner sep = 1pt] (1) {};
\draw[fill] (1, 0) circle (2pt) node[inner sep = 1pt] (2) {};
\draw[fill] (2, 0) circle (2pt) node[inner sep = 1pt] (3) {};
\draw[fill] (3, 0) circle (2pt) node[inner sep = 1pt] (4) {};
\draw[below] (1) node {$\scriptstyle 1$};
\draw[below] (2) node {$\scriptstyle 2$};
\draw[below] (3) node {$\scriptstyle 3$};
\draw[below] (4) node {$\scriptstyle 4$};
\draw (1) -- (2);
\draw (2) -- (3);
\draw (1) to[out = 45, in = 135] (3);
\draw (3) -- (4);
\end{tikzpicture}
\qquad\text{and}\qquad
\theta = \{(2, 1), (1, 3), (3, 2), (3, 4)\}
\end{equation}
we have
\[
([n], \theta) = 
\begin{tikzpicture}[scale = 0.5, baseline = 0.5*-0.3cm]
\draw[fill] (0, 0) circle (2pt) node[inner sep = 1pt] (1) {};
\draw[fill] (1, 0) circle (2pt) node[inner sep = 1pt] (2) {};
\draw[fill] (2, 0) circle (2pt) node[inner sep = 1pt] (3) {};
\draw[fill] (3, 0) circle (2pt) node[inner sep = 1pt] (4) {};
\draw[below] (1) node {$\scriptstyle 1$};
\draw[below] (2) node {$\scriptstyle 2$};
\draw[below] (3) node {$\scriptstyle 3$};
\draw[below] (4) node {$\scriptstyle 4$};
\draw[stealth-] (1) -- (2);
\draw[stealth-] (2) -- (3);
\draw[-stealth] (1) to[out = 55, in = 125] (3);
\draw[-stealth] (3) -- (4);
\end{tikzpicture}.
\]
Write $\OOO(\gamma)$ for the set of orientations of $\gamma$.  For $\theta \in \OOO(\gamma)$ and $i \in [n]$, say that the \emph{highest reachable vertex} from $i$ under $\theta$ is 
\[
\operatorname{hrv}(\theta, i) = \max\left\{j \in [n] \;\middle|\; \text{there is an \textit{increasing} path in $([n], \theta)$ from $i$ to $j$}  \right\}.
\]
For example, taking $\gamma$ and $\theta$ as in Equation~\eqref{eq:hrvexample}
\[
\operatorname{hrv}(\theta, 1) = 4, \qquad
\operatorname{hrv}(\theta, 2) = 2, \qquad
\operatorname{hrv}(\theta, 3) = 4, \qquad\text{and}\qquad
\operatorname{hrv}(\theta, 4) = 4.
\]
Finally, for $\theta \in \OOO(\gamma)$, the \emph{type} of $\theta$ is the partition $\operatorname{type}(\theta) \in \PPP_{n}$ obtained by truncating all zeros from the non-increasing re-ordering of the sequence
\[
\big( |\{i \in [n] \;|\; \operatorname{hrv}(\theta, i) = 1\}|, \ldots,  |\{i \in [n] \;|\; \operatorname{hrv}(\theta, i) = n\}| \big).
\]
For example, taking $\gamma$ and $\theta$ as in Equation~\eqref{eq:hrvexample}, $\operatorname{type}(\theta) = (3, 1)$.

\begin{thm}[{\cite[Theorem 2.9]{AS}}]
\label{thm:LLTepos}
For $n \ge 0$, let $\sigma \in \TTT\SSS_{n}$ and let $\gamma$ be the natural unit interval order on $[n]$ with edge set $E(\gamma) = \area(\sigma) \cup \diag(\sigma)$.  Then
\[
G_{\sigma}(\mathbf{x}; t) = \hspace{-1em} \sum_{\substack{ \text{$\diag(\sigma)$-ascending} \\ \theta \in \OOO(\gamma) }} \hspace{-1em} (t-1)^{\textstyle |\{ \scriptstyle \{i, j\} \in \area(\sigma) \;\textstyle|\scriptstyle\; \text{$(i, j) \in \theta$ with $i < j$} \textstyle\}|} e_{\operatorname{type}(\theta)}
\]
where the sum is over orientations $\theta \in \OOO(\gamma)$ with $(i, j) \in \theta$ for each  $i < j$  with $\{i, j\} \in \diag(\sigma)$.
\end{thm}

Evaluating the identity above at $t = q$, the expression $q-1$ can be interpreted as $|\FF_{q}^{\times}|$, the number of units in the field $\FF_{q}$.  As $|\FF_{q}^{\times}|$ is a positive integer, it can be interpreted as the multiplicity of a submodule, as will be discussed at the end of this section.

The \emph{Gelfand--Graev character} of $\GL_{n}$~\cite{GelGra} is the class function
\[
\bbGamma_{n} = \frac{1}{(q-1)^{n-1}} \ind^{\GL}_{\UT} (\psi^{E D^{n-1} S}),
\]
where $\psi^{E D^{n-1} S}$ is as defined in Section~\ref{sec:psisigma}; as the name suggests, $\bbGamma_{n} $ is actually a character of $\GL_{n}$; see Remark~\ref{rem:LLTnofactor}.  The \emph{degenerate Gelfand--Graev character}~\cite[\nopp 12]{Zel} indexed by a partition $\lambda = (\lambda_{1}, \ldots, \lambda_{\ell})$ is 
\[
\bbGamma_{\lambda} = \bbGamma_{\lambda_{1}} \cdots \bbGamma_{\lambda_{\ell}}.
\]

\begin{cor}
\label{cor:GLeposLLT}
For $n \ge 0$, let $\sigma \in \TTT\SSS_{n}$, and let $\gamma$ be the natural unit interval order on $[n]$ with edge set $E(\gamma) = \area(\sigma) \cup \diag(\sigma)$.  Then
\[
\ind^{\GL}_{\UT}(\psi^{\sigma}) = \hspace{-1em} \sum_{\substack{ \text{$\diag(\sigma)$-ascending} \\ \theta \in \OOO(\gamma) }} \hspace{-1em} (q-1)^{\textstyle |\{ \scriptstyle \{i, j\} \in E(\gamma) \;\textstyle|\scriptstyle\; \text{$(i, j) \in \theta$ with $i < j$} \textstyle\}|} \bbGamma_{\operatorname{type}(\theta)}
\]
where the sum is over orientations $\theta \in \OOO(\gamma)$ with $(i, j) \in \theta$ for each  $i < j$  with $\{i, j\} \in \diag(\sigma)$.
\end{cor}
\begin{proof}
Since the map $\canoUC$ restricts to an isomorphism from $\cfunsupp(\GL_{\bullet})$ to $\Sym$ (discussed in Section~\ref{sec:mainresult2}), and the involution $\omega$ is also an isomorphism, it is sufficient to establish that the above equation holds after the application of $\omega \circ \canoUC$ to both sides.  By Theorem~\ref{thm:LLTbijection} and Theorem~\ref{thm:LLTepos}, the left side becomes
\[
\omega \circ \canoUC \circ \ind^{\GL}_{\UT} (\psi^{\sigma}) =  \hspace{-1em} \sum_{\substack{ \text{$\diag(\sigma)$-ascending} \\ \theta \in \OOO(\gamma) }} \hspace{-1em} (q-1)^{\textstyle |\{ \scriptstyle \{i, j\} \in E(\gamma) \;\textstyle|\scriptstyle\; \text{$(i, j) \in \theta$ with $i < j$} \textstyle\}|} e_{\operatorname{type}(\theta)},
\]
so the claim will follow from the equation $\omega \circ \canoUC(\bbGamma_{n}) = e_{n}$.  This fact is know, but a short proof is included below for completeness.

Theorem~\ref{thm:LLTbijection} states that $\omega \circ \canoUC(\bbGamma_{n}) = G_{ED^{n-1}S}(\mathbf{x}; q)$.  With $\diag(ND^{n-1}S) = \{\{i, i+1\} \;|\; 1 \le i < n\}$, the definition of vertical strip LLT polynomials given in Section~\ref{sec:LLT} becomes
\[
G_{ED^{n-1}S}(\mathbf{x}; q) = \hspace{-0.75em} 
\sum_{\substack{\kappa: [n] \shortto \ZZ_{>0} \\ \kappa(1) < \cdots < \kappa(n) }}
\hspace{-0.75em}  x_{\kappa(1)} \cdots x_{\kappa(n)} = e_{n}. \qedhere
\]
\end{proof}

This result implies that the $\GL_{n}$-module affording $\ind^{\GL}_{\UT} (\psi^{\sigma})$ decomposes into a direct sum of submodules that each afford some degenerate Gelfand--Graev character.  Exhbiting this decomposition explicitly would give a new proof of Corollary~\ref{cor:GLeposLLT} and Theorem~\ref{thm:LLTepos}.

\begin{question}
Find a module theoretic proof of Corollary~\ref{cor:GLeposLLT}.
\end{question}

\printbibliography

\end{document}